\documentclass[10pt]{article}
\usepackage{amsmath}
\usepackage{mathrsfs}
\usepackage{amssymb}
\usepackage{latexsym}
\usepackage{amsthm}
\usepackage{amsmath,amssymb}
\usepackage{bm}
\usepackage{ascmac}
\usepackage{enumerate}
\theoremstyle{plain}
\usepackage[dvipdfmx]{graphicx}
\usepackage{tikz}
\usetikzlibrary{intersections,calc,arrows.meta}
\usepackage{float}

\usepackage[
  left=2.6cm,
  right=2.6cm,
  top=2.5cm,
  bottom=2.5cm]{geometry}

\newtheorem{thm}{Theorem}[section]

\newtheorem{prop}[thm]{Proposition}
\newtheorem{lem}[thm]{Lemma}
\newtheorem{cor}[thm]{Corollary}
\newtheorem{remark}[thm]{Remark}

\numberwithin{equation}{section}

\newcommand{\vare}{\varepsilon}

\allowdisplaybreaks

\title{Convergence of stationary distributions for 
a class of zero-range processes 
and its application to a wealth distribution model \\ with debt}
\author{
Hironobu Sakagawa
\thanks{
Department of Mathematics, 
Faculty of Science and Technology, Keio University, 
3-14-1 Hiyoshi, Kouhoku-ku, Yokohama 223-8522, JAPAN. 
{\it E-mail address\/}: sakagawa{\char'100}math.keio.ac.jp} 
}

\date{\empty}

\begin{document}

\maketitle 

\begin{abstract}
Understanding how macroscopic wealth distributions emerge from microscopic 
transaction rules among agents is a central problem in econophysics. 
While many traditional models restrict agents' wealth to non-negative values, 
debt is an essential feature of realistic economic systems. 
In this paper, we introduce a stochastic transaction mechanism 
into the agent-based model with a central-bank-mediated 
collective debt limit, as studied in \cite{LR19}. 
In our model, each agent probabilistically determines whether 
to transfer a coin, depending on their current asset or debt level. 
From a probabilistic perspective, the resulting system can be formulated 
as a zero-range process with $\mathbb{Z}$-valued site occupation numbers, 
where negative values represent debt. 
We prove that, under an appropriate scaling, 
the wealth distribution in the stationary state converges 
to either a two-sided Gamma distribution or a Gamma distribution, 
depending on the asymptotic behavior of the function governing coin transfers. 
Our results rigorously identify these previously unconsidered 
limiting distributions and provide a mathematical framework 
for understanding how stochastic monetary exchanges 
give rise to macroscopic wealth distributions.

\end{abstract}

\begin{center}{\small\bfseries Key words.}
{\small Econophysics, zero-range process, 
stationary distribution, equivalence of ensembles, 
local limit theorem.} 
\end{center}


\section{Model and result}

\subsection{Introduction}
In economic systems, understanding how macroscopic wealth distributions 
and economic inequality emerge from microscopic transaction rules 
among agents is one of the most fundamental challenges 
in the field of econophysics.
Inspired by the close analogy with energy transfer between particles 
in statistical mechanics, mathematical models in which individual 
agents randomly exchange wealth or money have been actively studied. 
As a result, numerous non-trivial macroscopic behaviors arising from 
various transaction rules have been revealed 
through numerical simulations and heuristic physical arguments 
(cf. \cite{CCC}, \cite{GG24}, \cite{YR09} and references therein).

To establish a mathematically rigorous analytical framework, 
\cite{L17} and \cite{LR18} introduced 
spatially explicit discrete models incorporating 
local interactions on general finite connected graphs 
representing social networks. 
By formulating these models as Markov chains, 
they characterized their stationary distributions 
and investigated the convergence of wealth distributions 
through heuristic calculations 
(see also \cite[Section 7.5.3]{L24} for a summary of 
mathematical literature in econophysics). 
More recently, in \cite{S26}, we formulated generalized versions 
of the immediate exchange model, the uniform reshuffling model, 
and the uniform saving model studied in \cite{L17} and \cite{LR18} 
as discrete-time interacting particle systems 
and characterized their stationary distributions. 
Moreover, under an appropriate scaling, we provided a mathematically 
rigorous proof of the convergence of wealth distributions, 
thereby establishing, at a mathematical level, 
the universality of wealth distributions independently of the 
initial configurations and network structures. 

However, many models in econophysics assume that the money 
held by agents remains non-negative. 
In real-world economic systems, the presence of debt in the form 
of borrowing or credit is inevitable. 
To incorporate this aspect, \cite{LR19} and \cite{XDW05} considered models 
in which agents are permitted to incur debt within bank-prescribed 
limits while engaging in economic activities, 
such as  exchanging coins among agents. 
Specifically, the time evolution rule studied in \cite{LR19} is 
structured as follows. 
The model consists of a system of economic agents, 
each characterized by their monetary holdings, or number of coins 
at a given time. 
The process evolves in discrete time. 
At each time step, two agents are selected uniformly at random from 
the population, and one transfers a single coin to the other. 
Even if an agent holds no coins, they are permitted to borrow 
from a central bank in order to carry out the transfer. 
However, the bank operates under a fixed total credit capacity, 
so that the aggregate amount of debt incurred by all agents 
cannot exceed this bound. 
\cite{LR19} formulated this dynamics as a Markov chain, 
characterized its stationary distribution, 
and obtained an explicit elementary expression for it. 
Furthermore, assuming that the marginal distribution 
of the stationary distribution converges to an 
asymmetric Laplace distribution in the infinite-volume limit, 
they studied this limiting distribution and determined its parameters. 
A generalization to multiple banks is also considered in \cite{LR24}. 

The objective of the present paper is to further generalize 
the model introduced by \cite{LR19} and, 
in a spirit similar to \cite{S26}, 
to provide a mathematically rigorous proof of the convergence 
of its limiting wealth distribution under an appropriate scaling. 
As a key generalization of the model, 
we incorporate a stochastic mechanism into the transaction process, 
whereas previous studies assumed that coin transfers between 
two selected agents were always executed 
whenever bank borrowing was permitted. 
Specifically, the transferring agent probabilistically determines  
whether to transfer a coin, 
depending on their current asset or debt level. 
This generalization appears natural from both mathematical 
and economic perspectives.
In particular, from a probability-theoretic viewpoint, 
this model can be viewed as a generalization of the zero-range process. 
Introduced in \cite{S70}, the zero-range process is a 
classical model of interacting 
particle systems consisting of indistinguishable particles 
that randomly jump from site to site on a lattice. 
The rate at which a particle jumps depends solely on the number 
of particles present at its current location. 
By identifying coins with particles, 
the wealth distribution model described above 
naturally falls within this framework. 
More precisely, our model can be viewed as a 
zero-range-type process with integer valued occupation variables, 
where negative occupation numbers represent debts. 
In contrast to the conventional zero-range process, 
in which the occupation variables are restricted to $\mathbb{Z}_{+}$, 
our model allows them to take values in $\mathbb{Z}$.
While the zero-range process is a well-studied interacting particle system 
with widespread applications in various mathematical 
and physical contexts, standard formulations have been restricted 
to non-negative site occupations. 
Thus, zero-range-type models with integer valued site occupations, 
where negative values represent debts, remain largely unexplored. 
This provides a natural probabilistic framework 
for investigating the effects of probabilistic decision-making 
and debt on the macroscopic distribution of wealth. 

\subsection{Model description and result}

Let us state our model precisely. 
In the following, $\mathbb{Z}_+=\{0, 1, 2, \cdots \}$ denotes the set of 
non-negative integers and 
$\mathbb{N}=\{1, 2, 3, \cdots \}$ denotes the set of positive integers. 
Set $\Lambda_N = \{1, 2, \cdots, N\}$, $N\in \mathbb{N}$. 
Each site $x\in \Lambda_N$ corresponds to an agent and 
we assume that the economy can be represented 
by the set of agents $\Lambda_N$. 
$N$ corresponds to the population size. 
For $K, L \in \mathbb{Z}_+$, 
we define the configuration space 
\begin{equation*}
\Omega_N^{K, L}=\bigl\{
\xi = \{\xi(x)\}_{x\in \Lambda_N} \in \mathbb{Z}^{\Lambda_N}; 
\sum\limits_{x\in \Lambda_N}\xi(x) =K, 
\sum\limits_{x\in \Lambda_N} (-\xi(x))I(\xi(x) \leq -1) 
\leq L 
\bigr\}.
\end{equation*}
The amount of wealth or money each agent holds 
is represented by the number of coins, and 
for each $\xi \in \Omega_N^{K, L}$ and $x\in \Lambda_N$, 
we interpret $\xi(x) \in \mathbb{Z}$ not as the number of particles, 
but as the number of coins held by agent $x$, 
where negative values represent debt held by agent $x$. 
$\Omega_N^{K, L}$ represents the set of all configurations 
such that the total wealth is fixed at $K$ and the total debt does 
not exceed $L$. 
$L$ represents the total amount of coins that the bank can lend. 

Next, we introduce the time evolution on $\Omega_N^{K, L}$. 
Let $\Lambda_N^*= \{(x, y); x,y \in \Lambda_N, x\ne y\}$ 
be the set of all directed edges on 
$\Lambda_N$. 
$(x, y)\in \Lambda_N^*$ represents the directed edge 
from $x\in \Lambda_N$ to $y\in \Lambda_N$. 
$\mathcal{P}=\{p(x, y); (x, y) \in \Lambda^*_N\}$ 
denotes a probability distribution on $\Lambda_N^*$ 
and this represents the distribution that determines 
which two agents transfer a coin at each time step. 
In particular, 
$\widetilde{\mathcal{P}}=
\{\frac{1}{\pi(x)}p(x, y); (x, y) \in \Lambda^*_N\}$ 
gives a transition probability of a random walk on 
$\Lambda_N$ where 
$\pi(x)= \sum\limits_{\substack{y\in \Lambda_N \\ y\ne x}}
p(x, y)$, $x\in \Lambda_N$. 
We also take a function $g : \mathbb{Z} \to (0, 1]$. 
For $k \in \mathbb{Z}$, $g(k)$ gives the probability that 
an agent holding $k$ coins passes a coin. 
Then, let $\{X_n\}_{n \geq 0}$ denote the time-homogeneous 
Markov chain on the state space $\Omega_N^{K, L}$. 
Here, given $X_n=\eta \in \Omega_N^{K, L}$, 
the configuration $X_{n+1}$ is determined according to the following rule:
\begin{enumerate} 
\item
Choose a directed edge $(x, y) \in \Lambda_N^*$ 
according to the distribution $\mathcal{P}$. 
\item
Update $X_n$ to $X_{n+1}$ as follows. 
\begin{equation*}
X_{n+1} = 
\begin{cases}
\eta^{x \to y} & \text{ with probability } g(\eta(x)) \text{ if } 
\eta^{x \to y} \in \Omega_N^{K, L}, \\
\eta & \text{ with probability } 1-g(\eta(x)) \text{ if } 
\eta^{x\to y} \in \Omega_N^{K, L}, \\
\eta & \text{ if } 
\eta^{x\to y} \notin \Omega_N^{K, L}. 
\end{cases}
\end{equation*}
\noindent
Here, for given the configuration $\eta\in \mathbb{Z}^{\Lambda_N}$ and 
$x, y\in \Lambda_N$, 
the configuration $\eta^{x\to y}$ is defined by the following. 
\begin{align*}
\eta^{x\to y}(z)= 
\begin{cases}
\eta(x)-1 & \text{ if } z=x, \\
\eta(y)+1 & \text{ if } z=y, \\
\eta(z) & \text{ otherwise}. 
\end{cases}
\end{align*}
\end{enumerate}
This dynamics is exactly a discrete time zero-range-type process 
and can be interpreted in the context of the wealth distribution model 
as follows. 
At each time step, a transfer of a single coin is 
attempted along a randomly chosen directed edge $(x, y) \in \Lambda_N^*$. 
Agent $x$ decides whether to pass one coin to agent $y$ 
with a probability depending on $\eta(x)$, 
the number of coins held by agent $x$. 
Even if agent $x$ holds no coins, 
agent $x$ can transfer a coin to agent $y$ 
by borrowing from the bank. 
However, if passing one coin from agent $x$ to agent $y$ would 
cause the total debt of all agents to exceed $L$, 
the transfer does not take place. 
In this process, the total wealth is conserved 
and the total debt does not exceed $L$. 
It is worth noting that the model studied in \cite{LR19} 
corresponds to the case where $g \equiv 1$.

We first characterize the stationary distributions 
of this Markov chain. 
In the following, 
for a function $g$ on $\mathbb{Z}$ and $k\in \mathbb{Z}$, 
we denote $g(k)!$ by $g(k)! = \prod\limits_{i=1}^k g(i)$ if $k\geq 1$ and 
$g(k)! = \prod\limits_{i=k}^0 g(i)$ if $k\leq 0$. 
Also, we interpret $\frac{1}{g(0)!}$ as $1$ for notational simplicity. 
\begin{prop}\label{prop1}
Let $N \in \mathbb{N}$ and $K, L \in \mathbb{Z}_+$ be fixed. 
We assume that $p$ is symmetric, that is, $p(x, y)=p(y, x)$ 
for every $x, y\in \Lambda_N$ with $x\ne y$, 
and the random walk on $\Lambda_N$ 
determined from the transition probability $\widetilde{\mathcal{P}}$ is 
irreducible. 
Then, there is a unique stationary distribution $\mu_{N}^{K, L}$ 
for $\{X_n\}_{n \geq 0}$ and it is given by 
\begin{equation}\label{stat}
\mu_{N}^{K, L}(\xi) = \frac{1}{Z_{N}^{K, L}}
\prod_{x\in \Lambda_N}\!\! G(\xi(x)),\ 
\xi \in \Omega_N^{K, L}, 
\end{equation}
where 
\begin{equation}\label{margi}
G(k) = 
\begin{cases}
\frac{1}{g(k)!} & \text{ if }\ k \geq 0, \\
{g(k+1)!} & \text{ if }\ k \leq -1, \\
\end{cases}
\end{equation} 
and 
$Z_{N}^{K, L}= \sum\limits_{\xi \in \Omega_N^{K, L}}
\bigl\{\prod\limits_{x\in \Lambda_N} G(\xi(x))\bigr\}$ 
is the normalization factor. 
\end{prop}
\noindent
Note that the choice of $\mathcal{P}$ does not affect 
the stationary distribution. 
Also, when $g \equiv 1$, the function $G$ 
is identically equal to $1$, and (\ref{stat}) coincides with 
the uniform distribution on $\Omega_N^{K, L}$, 
which agrees with the result in \cite{LR19}. 

\begin{remark}
Although we consider the discrete time dynamics, 
the same result holds in the continuous time dynamics. 
That is, consider an infinitesimal generator given by the following. 
\begin{align*}
\mathcal{L}_Nf(\eta)= 
\sum\limits_{(x, y)\in \Lambda_N^*}p(x, y) g(\eta(x))
I(\eta^{x\to y}\in \Omega_N^{K, L})\bigl(
f(\eta^{x\to y}) -f(\eta)\bigr), \ \ \eta\in \Omega_N^{K, L}. 
\end{align*}
Then, the stationary distribution of the corresponding 
continuous time Markov process on 
$\Omega_N^{K, L}$ is given by (\ref{stat}). 
In particular, $g$ does not need to be less than $1$ in this case.
\end{remark}

\begin{remark}
If we assume that $g\equiv 0$ on $\mathbb{Z}_-$, 
then $\{X_n\}_{n\geq 0}$ is the standard zero-range process on 
$\Omega_N^{K,0}$ with a total of $K$ particles, 
and the stationary distribution $\mu_N^{K,0}$ given by (\ref{stat}) 
coincides with the well-known stationary distribution for the 
zero-range process (cf. \cite{KL}).
\end{remark}

Now, we are in the position to state the main result of this paper. 
We assume the following condition for the function $g: \mathbb{Z}\to (0, 1]$.  

\medskip

\noindent
$(g)$ \ 
The function $g$ is eventually monotone on $\mathbb{Z}_+$ and $\mathbb{Z}_-$, 
respectively, and there exist $g^+\in (0, 1]$, $g^-\in (0, 1]$, $\rho>-1$ 
such that 
\begin{align*}
g(k)= g^+\bigl(1-\frac{\rho}{k}+O(\frac{1}{k^2})\bigr) 
\text{ as } k\to\infty 
\text{ and }
g(k)= g^-\bigl(1-\frac{\rho}{k}+O(\frac{1}{k^2})\bigr) 
\text{ as } k\to -\infty. 
\end{align*}

\noindent
Conditions of this kind also arise in studies of condensation 
in the zero-range process 
(cf. \cite{AGL13}, \cite{GSS03}, etc.).
In particular, the condition $(g)$ yields that there exist 
constants $C^+>0$, $C^->0$ such that the following holds. 
\begin{align}\label{const}
\frac{(g^+)^k}{g(k)!} = C^+ k^{\rho}(1+o(1))\ 
\text{ and }\ 
\frac{g(-k+1)!}{(g^-)^k} = C^- k^{\rho}(1+o(1)) 
\text{ as } k\to \infty. 
\end{align}
Next, to study the infinite-volume limit of the system 
under the stationary distribution $\mu_N^{K, L}$, 
we consider the joint limit where both 
the population size $N$ and the average number of 
coins per agent $\frac{K}{N}$ diverge to infinity. 
By analogy with the thermodynamic limit in statistical mechanics, 
these two limits are referred to as the large population limit 
and the large money temperature limit, respectively, 
within the framework of econophysics. 
To mathematically justify this limiting procedure, 
following \cite{S26}, 
we fix constants $\kappa, r > 0$ and a sequence of positive real 
numbers $\{a_N\}$ diverging to infinity, 
and set $K_N = \kappa a_N N$ and $L_N = r \kappa a_N N$. 
We then consider the limit of the distribution of the field 
scaled by a factor of $\frac{1}{a_N}$ under $\mu_N^{K_N, L_N}$. 
In particular, since 
$\frac{1}{N}\sum\limits_{x\in \Lambda_N} \bigl( \frac{1}{a_N}\xi(x) \bigr) 
= \frac{K_N}{N a_N} = \kappa$ for $\xi \in \Omega_N^{K_N, L_N}$, 
$\kappa$ represents the average number of coins per agent 
of the scaled field, 
while $r = \frac{L_N}{K_N}$ denotes the ratio of 
the maximum number of coins the bank can lend 
to the total number of coins held by all agents. 

Under the above setting, we can prove that the law of the scaled field 
$\bigl\{ \frac{1}{a_N}\xi(x)\bigr\}_{x\in \Lambda_N}$ under 
$\mu_{N}^{K, L}$ converges 
to the i.i.d. product of 
probability distributions on $\mathbb{R}$. 
Its marginal distribution varies 
significantly depending on the 
asymptotic behavior of the function $g$. 
In the following, 
a function $f$ on $\mathbb{R}^{\mathbb{N}}$ is 
called local if it depends only on finitely many coordinates. 
\begin{thm}\label{thm1}
Assume the condition $(g)$. 
Set $K_N=\kappa a_N N$ and $L_N=r \kappa a_N N$
where $\kappa, r>0$ and 
$\{a_N\}_{N\geq 1}$ be a sequence which satisfies 
$\lim\limits_{N\to \infty}a_N=\infty$ and 
$a_N=o(N^\beta)$ for some $\beta>0$. 
Then, for every bounded continuous local function 
$f:\mathbb{R}^A \to \mathbb{R}$, it holds that 
\begin{align*}
\lim_{N\to \infty}
E^{\mu_{N}^{K_N, L_N}} \Bigl[ f\bigl( 
\frac{\cdot}{a_N} \bigr)\Bigr] 
= 
E^{\overline{\mu}_{\kappa, r}^A}\bigl[f(\, \cdot\, )\bigr], 
\end{align*}
where $A$ is a finite subset of $\mathbb{N}$ and 
$\overline{\mu}_{\kappa, r}^A$ denotes 
the product probability measure on 
$\mathbb{R}$ whose one site marginal distribution $\mu_{\kappa, r}$ 
is given by the following. 
\begin{enumerate}[$(i)$]
\item
Case $g^+ \leq g^-$. 
\begin{align*}\label{limdist}
\mu_{\kappa, r}(dt) = 
\frac{1}{\Gamma(\rho +1) \{C^+ (\theta^+)^{\rho +1}
+ C^- (\theta^-)^{\rho +1}\}}
\Bigl\{ C^+ t^\rho e^{-\frac{1}{\theta^+}t}  1_{\{t \geq 0\}} + 
C^- |t|^\rho e^{-\frac{1}{\theta^-}|t|} 
1_{\{t < 0\}} \Bigr\}dt, 
\end{align*}
where $C^+, C^-$ are positive constants determined from the 
relation (\ref{const}) and 
$(\theta^+, \theta^-)$ is the unique solution of the following equation 
in $(0, \infty)\times (0, \infty)$. 
\begin{equation}\label{lim}
\begin{cases}
& \frac{(\rho+1) C^+(\theta^+)^{\rho + 2}}
{C^+(\theta^+)^{\rho + 1}+ C^-(\theta^-)^{\rho + 1}}
= (1+r)\kappa, \\
& \frac{(\rho+1) C^-(\theta^-)^{\rho + 2}}
{C^+(\theta^+)^{\rho + 1}+ C^-(\theta^-)^{\rho + 1}}
= r \kappa.
\end{cases}
\end{equation}
$\Gamma(\beta)= \int_0^\infty t^{\beta-1}e^{-t}dt$, $\beta>0$ 
denotes the Gamma function. 
\item
Case $g^+ > g^-$. 
\begin{align*}
\mu_{\kappa, r}(dt) = 
\frac{1}{\Gamma(\rho +1) \theta^{\rho +1}}
t^\rho e^{-\frac{1}{\theta}t} 1_{\{t \geq 0\}} dt, 
\end{align*}
where $\theta = \frac{\kappa }{\rho +1}$. 
\end{enumerate}
\end{thm}
\noindent
Note that $\mu_{\kappa, r}$ does not depend on $r$ 
when $g^+>g^-$. 

As an easy consequence of Proposition \ref{prop1} 
and Theorem \ref{thm1}  we obtain the following. 
For a finite set $A$, we denote the number of elements 
in $A$ by $|A|$. 
\begin{cor}\label{cor1}
Under the assumptions in Proposition \ref{prop1} and 
Theorem \ref{thm1}, we have 
\begin{align*}
\lim_{N\to \infty}\lim_{n\to \infty}
\max\limits_{\eta \in \Omega_N^{K_N, L_N}}
\Bigl\{
E_{\eta} \Bigl[ \frac{1}{N}
\bigm|\!\!\bigl\{x\in \Lambda_N; 
\frac{1}{a_N} {X_n (x)} \in (b, c) \bigr\}
\!\!\bigm|\Bigr] \Bigr\}
= \mu_{\kappa, r}((b, c)), 
\end{align*}
for every $-\infty \leq b < c \leq \infty$ 
where $E_{\eta}[\ \cdot \ ]$ denotes the expectation 
with respect to the law of the Markov chain with the 
initial condition $\eta$. 
\end{cor}
Corollary \ref{cor1} demonstrates the convergence of 
the asymptotic wealth distribution, that is, 
the proportion of agents holding 
a specific number of coins or a specific amount of debt, 
as time approaches infinity, and in the large population and the 
large money temperature limits. 
Under the condition $(g)$, 
when $g^+ \leq g^-$, the support of the limiting distribution 
is the entire real line $\mathbb{R}$, 
whereas when $g^+ > g^-$, its support is restricted to the positive side 
$\mathbb{R}_+$. 
In the latter case, this implies that the proportion of 
agents holding debt becomes zero. 
Intuitively, this can be understood as follows: 
when $g^+ > g^-$, agents in debt are less likely to transfer 
coins compared to those holding asset (positive number of coins). 
Consequently, coins flow more readily from coin holders 
to agents in debt, implying that debt is eliminated 
in the stationary state under the scaling limit.
On the other hand, when $g^+ \leq g^-$, 
it follows from (\ref{lim}) that the mean of the 
negative part under $\mu_{\kappa,r}$ is given by
$\int_{-\infty}^0 (-t)\mu_{\kappa,r}(dt)=r\kappa$. 
This means that, when agents in debt are at least as likely as coin holders 
to transfer a coin, the stationary state corresponds to a 
situation in which the system as a whole accumulates 
the maximum possible amount of debt. 
We also note that the limiting distribution is either 
a two-sided Gamma distribution or a Gamma distribution, 
and their shape parameters are determined by 
the second leading-order term of the limit of 
the function $g$. 

When $g$ is constant on the positive and negative sides 
respectively, the condition $(g)$ holds with $\rho = 0$, 
and the limiting distribution becomes an 
asymmetric Laplace distribution or an exponential distribution. 
Specifically, consider the case where $g \equiv 1$. 
The condition $(g)$ is satisfied with $g^+ = g^- = 1$, $\rho = 0$ 
and $C^+ = C^- = 1$. 
Then, the limiting distribution $\mu_{\kappa, r}$ yields 
an asymmetric Laplace distribution whose 
positive-side and negative-side parameters are 
given by 
$\frac{1}{\theta^+}=\frac{1}{\kappa}
\bigl(1-\sqrt{\frac{r}{r+1}}\bigr)$ and 
$\frac{1}{\theta^-}=\frac{1}{\kappa}
\bigl(\sqrt{\frac{r+1}{r}}-1\bigr)$, respectively. 
This agrees with the heuristic calculations given in \cite[Section 8]{LR19}. 
Our result rigorously formulates and generalizes 
earlier studies \cite{LR19}, \cite{XDW05} 
which relied on numerical simulations and heuristic arguments, 
thereby settling their conjecture with a mathematically rigorous proof. 
Furthermore, all previous works focused solely on the case $g \equiv 1$. 
We demonstrate for the first time that 
introducing probabilistic decision-making 
based on the asset or debt level of each agent leads to 
a two-sided Gamma distribution or a Gamma distribution 
as the limiting distribution, 
which is consistent with numerical simulations. 

To illustrate these theoretical results, 
we present in Figure 1 numerical simulation results for a single realization 
of the Markov chain $\{X_n\}_{n\geq 0}$. 
We set the number of agents to $N=10^4$, 
the total number of coins to $K=10^6$, 
and the maximum amount of coins that the bank can lend to $L=10^5$. 
Thus, the average number of coins per 
agent is $\kappa =100$, and $r=\frac{L}{K}=0.1$.  
The initial condition is set to a constant configuration 
$X_0 \equiv 100$. 
We take $\mathcal{P}$ to be the uniform distribution on $\Lambda_N^*$, 
namely, each $(x, y)\in \Lambda_N^*$ is chosen with 
probability $\frac{1}{N(N-1)}$. 
$g^+$ and $g^-$ be either $0.6$ or $0.7$ and 
we consider the following two cases for the function $g$. 
\begin{itemize}
\item{Case (i)}:\ 
$g(k) = g^+$ for $k\geq 1$ and $g(k)=g^-$ for $k\leq 0$. 
\item{Case (ii)}:\ 
$g(k) = g^+\bigl(1-\frac{2}{k}\bigr)$ for $k\geq 7$, 
$g(k)=g^+$ for $k\in \{1, 2, \cdots, 6\}$, 
$g(k) = g^-\bigl(1-\frac{2}{k}\bigr)$ for $k\leq -5$ and 
$g(k)=g^-$ for $k\in \{-4, -3, \cdots, 0\}$. 
\end{itemize}
In particular, the condition $(g)$ holds with 
$\rho=0$ and $C^+=C^-=1$ in Case (i), 
and $\rho=2$ and $C^+=C^-=\frac{1}{30}$ in Case (ii). 
In Case (ii), the function $g$ is also defined so that 
$g(k)\in(0,1)$ for every $k\in\mathbb{Z}$. 
When $C^+=C^-$, 
${\theta^+}$ and ${\theta^-}$ in (\ref{lim}) are 
given by 
${\theta^+}=\frac{\kappa}{\rho+1}
\bigl\{1+r + r^{\frac{\rho+1}{\rho+2}}(1+r)^{\frac{1}{\rho+2}}\bigr\}$ 
and 
${\theta^-}=\frac{\kappa}{\rho+1}
\bigl\{r + r^{\frac{1}{\rho+2}}(1+r)^{\frac{\rho+1}{\rho+2}}\bigr\}$, 
respectively. 

\begin{figure}[H]
\begin{tabular}{cc}
\begin{minipage}[t]{0.5\hsize}
\centering
\includegraphics[keepaspectratio, scale=0.35]{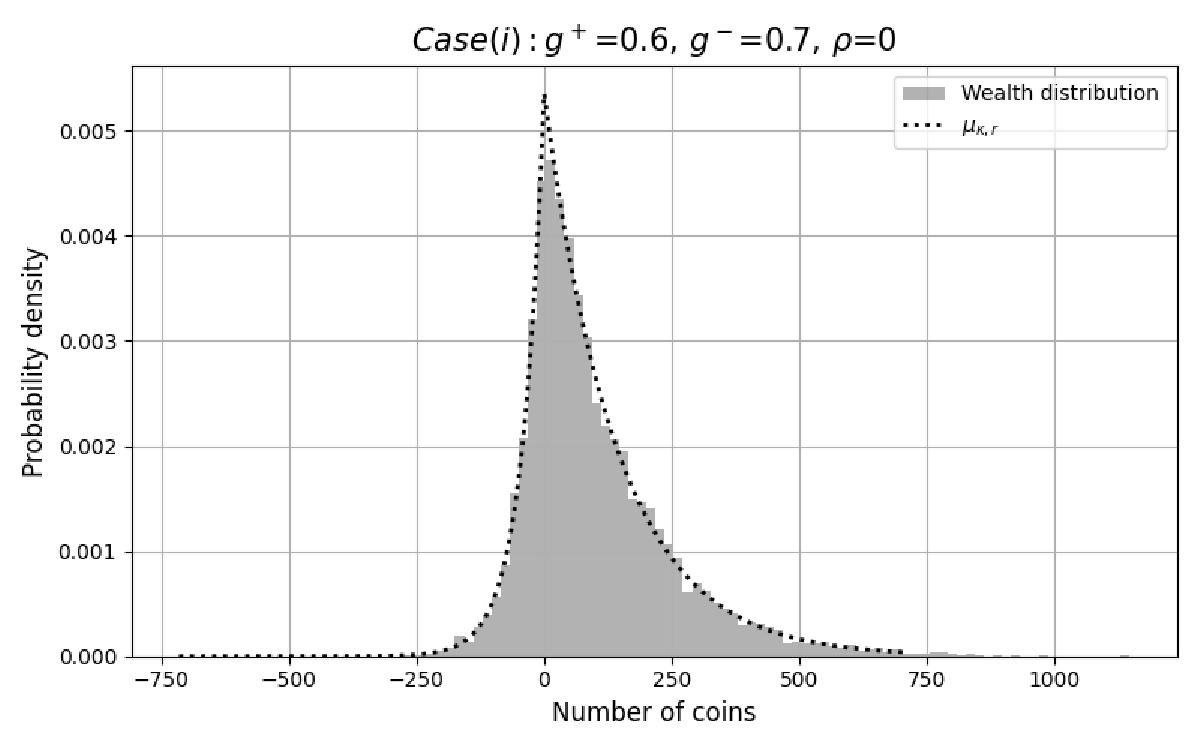}
\end{minipage} &
\begin{minipage}[t]{0.5\hsize}
\includegraphics[keepaspectratio, scale=0.34]{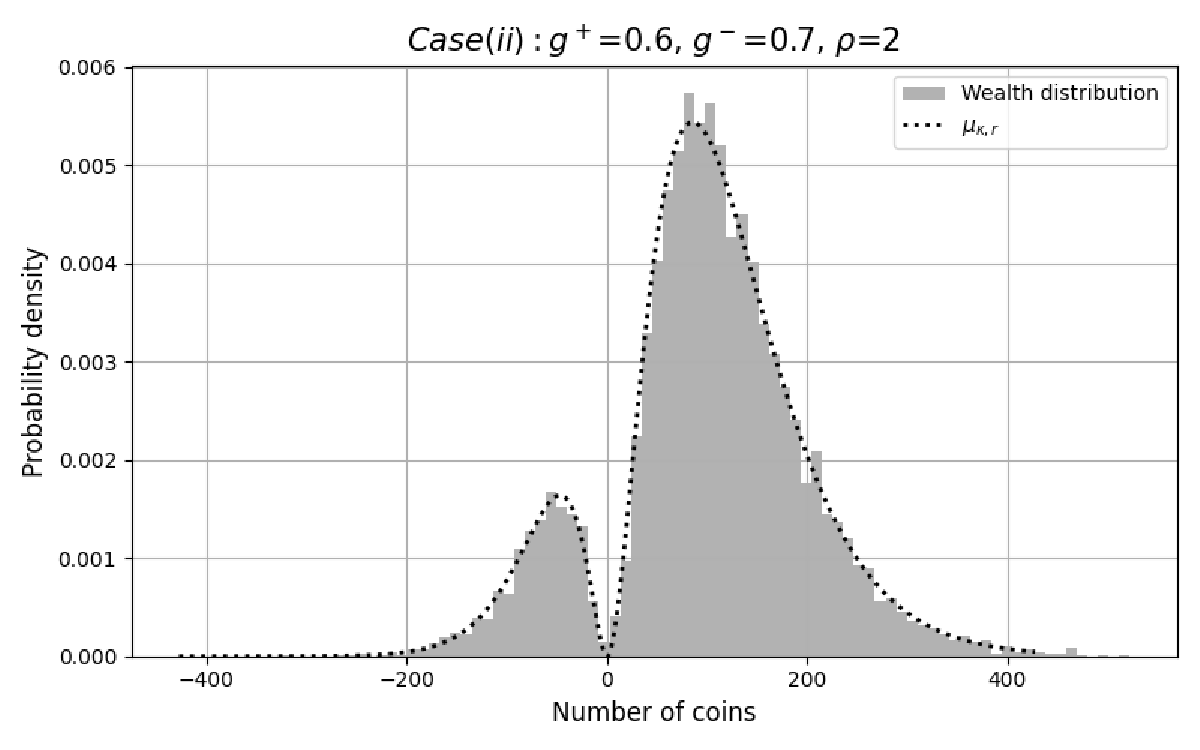}        
\end{minipage} \\
\begin{minipage}[t]{0.5\hsize}
\centering
\includegraphics[keepaspectratio, scale=0.34]{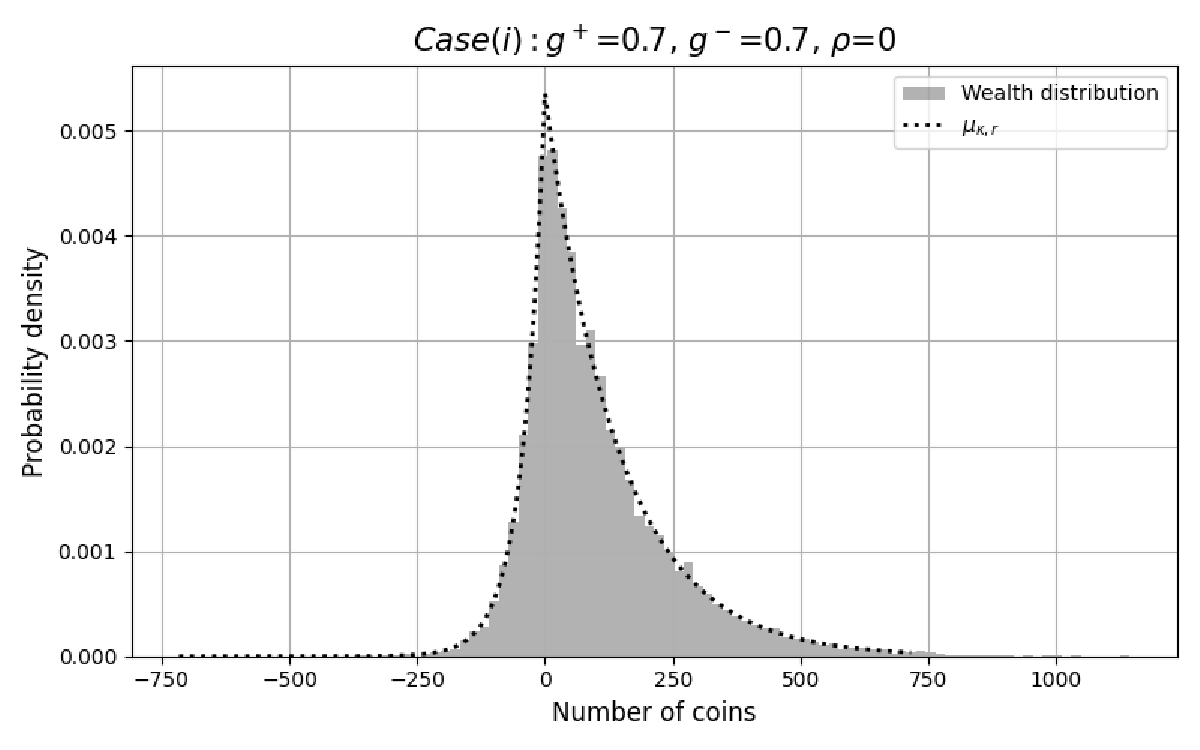}        
\end{minipage} &
\begin{minipage}[t]{0.5\hsize}
\includegraphics[keepaspectratio, scale=0.34]{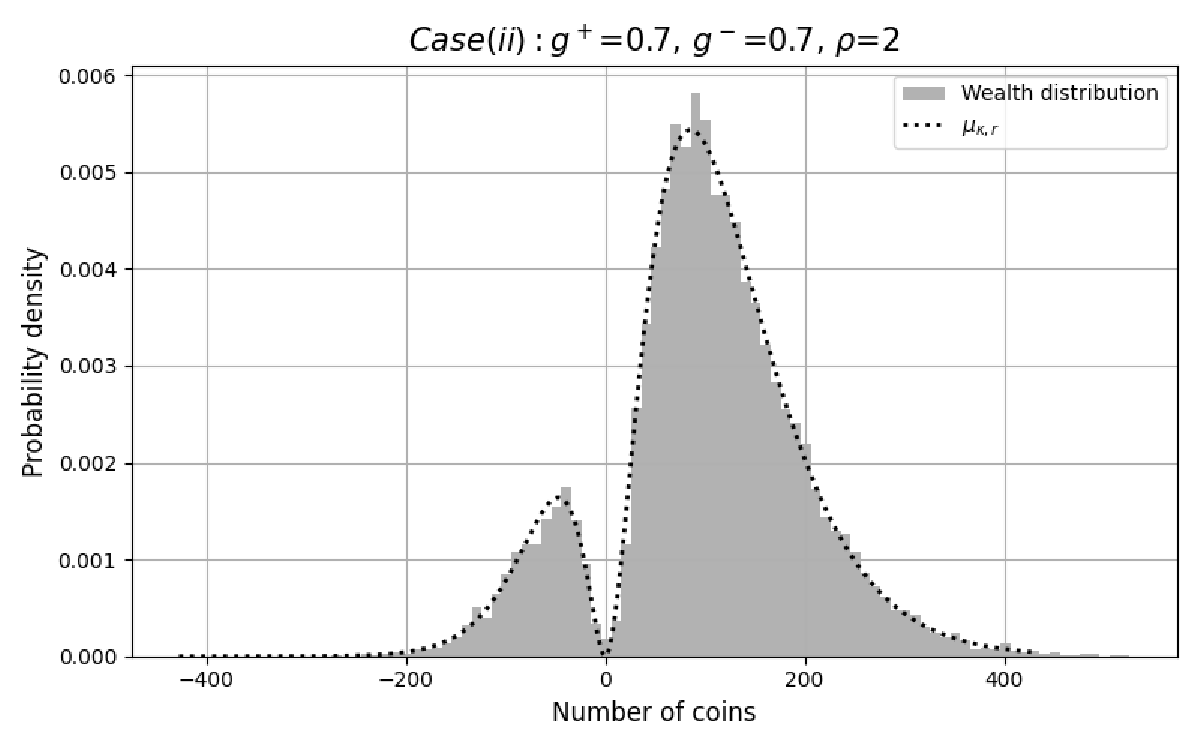}        
\end{minipage} \\
\begin{minipage}[t]{0.5\hsize}
\centering
\includegraphics[keepaspectratio, scale=0.34]{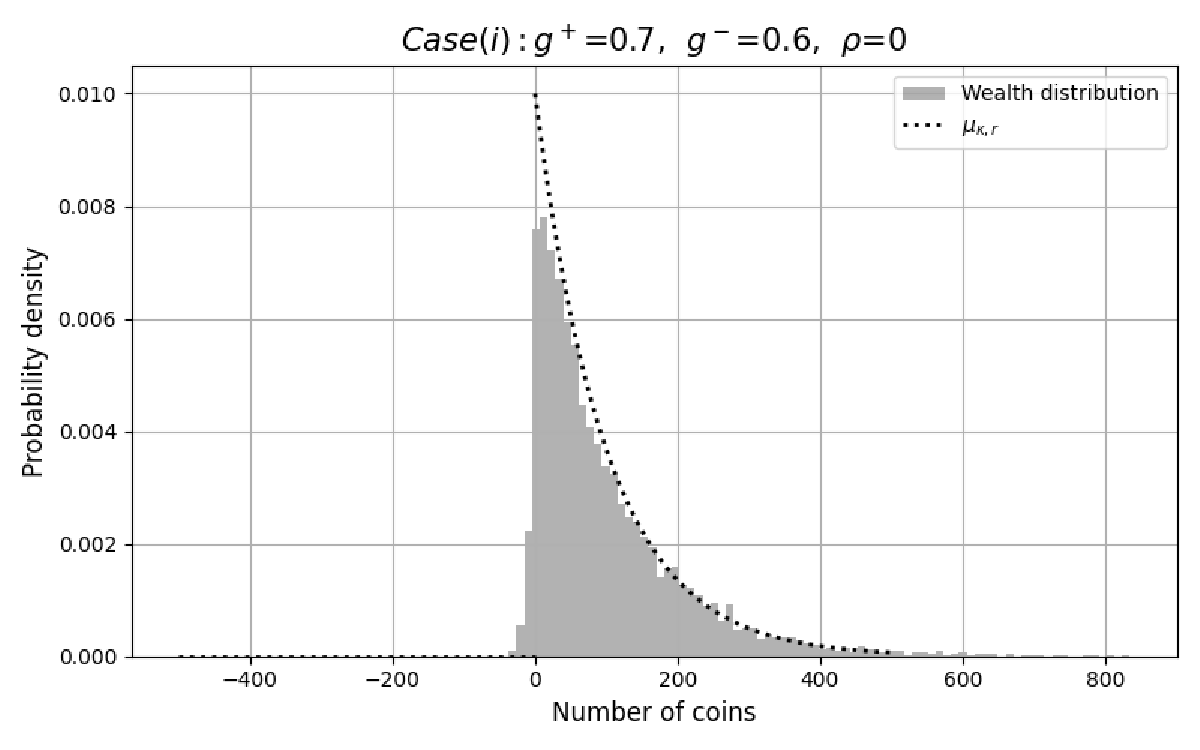}        
\end{minipage} &
\begin{minipage}[t]{0.5\hsize}
\includegraphics[keepaspectratio, scale=0.34]{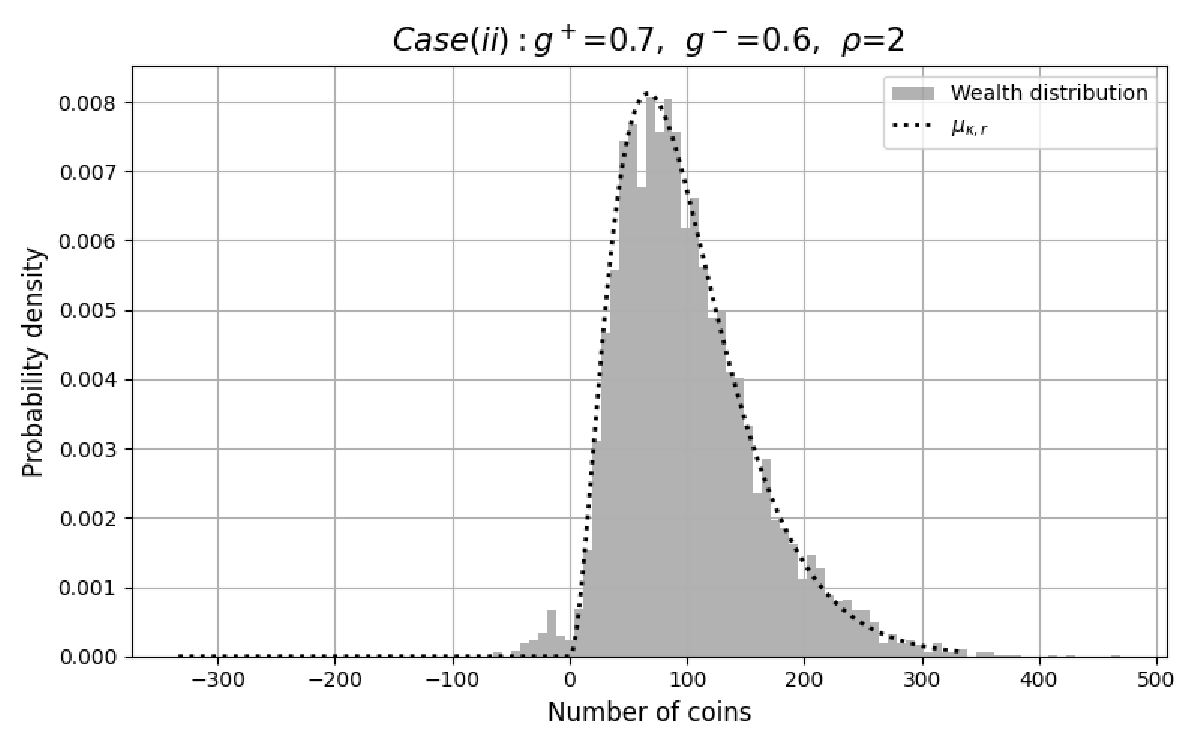}        
\end{minipage} 
\end{tabular}
\caption{Simulation results for a 
single realization of the Markov chain $\{X_n\}_{n\geq 0}$. 
The left and right columns show the simulation results 
for Case (i) and Case (ii), respectively. 
From top to bottom, the values of $(g^+, g^-)$ 
are $(0.6, 0.7)$, $(0.7, 0.7)$, and $(0.7, 0.6)$. 
The gray histograms represent 
the wealth distribution, i.e. 
the proportion of agents holding 
a specific number of coins or a specific amount of debt 
after $n=10^9$ updates. 
The dotted line represents the probability density function 
corresponding to $\mu_{\kappa,r}$ given in Theorem \ref{thm1}. 
The top two rows show the probability density functions 
of the asymmetric Laplace distribution or the two-sided Gamma distribution, 
while the bottom row shows those of the exponential distribution 
or the Gamma distribution. 
Simulations were implemented in Python.
}
\end{figure}

Next, we give some comments about the strategy of the proof. 
The convergence of the marginal distributions of $\mu_{N}^{K, L}$ 
is closely related to 
the equivalence of ensembles (cf. \cite{FV}, \cite{KL}). 
The main difference from the usual equivalence of ensembles is that, 
in addition to the conservation of the total quantity 
$\sum\limits_{x\in\Lambda_N}\xi(x)$, 
we impose the constraint 
$\sum\limits_{x\in\Lambda_N}(-\xi(x))I(\xi(x)\leq -1)\leq L_N$. 
In the proof of the equivalence of ensembles in our setting, 
it is important to consider the exponential family of probability 
distributions with weight $G$ given in (\ref{margi}). 
When $g^+>g^-$, 
there exists $\phi>0$ such that 
$\sum\limits_{k\in\mathbb{Z}}\phi^k G(k)<\infty$. 
Under the corresponding grand canonical measure, we can show that 
the constraint 
$\sum\limits_{x\in\Lambda_N}(-\xi(x))I(\xi(x)\leq -1)\leq L_N$ 
has an asymptotically negligible effect on the probability that 
$\sum\limits_{x\in\Lambda_N}\xi(x)=K_N$. 
Consequently, 
noting that the scaling in our setting differs from the usual one, 
the proof of equivalence of ensembles 
for the canonical measures conditioned on 
$\sum\limits_{x\in\Lambda_N}\xi(x)$ can be carried out 
in a manner similar to that in \cite{S26}. 
On the other hand, 
when $g^+ \leq g^-$, 
there is no $\phi>0$ satisfying 
$\sum\limits_{k\in\mathbb{Z}}\phi^k G(k)<\infty$. 
To address this difficulty, 
in this case, we consider a grand canonical measure 
obtained by applying different exponential tiltings 
to the positive and negative parts, separately. 
In addition, 
a simple but crucial observation is that $\mu_N^{K,L}$ 
can be represented as a mixture of canonical measures conditioned 
on the two quantities $\sum\limits_{x\in\Lambda_N}\xi(x)$ and 
$\sum\limits_{x\in\Lambda_N}(-\xi(x))I(\xi(x)\leq -1)$,
see (\ref{dec}) below. 
Then, Theorem \ref{thm1} is proved by combining detailed estimates 
of the corresponding weights with the proof of 
the equivalence of ensembles 
for canonical measures conditioned on these two quantities.
In particular, as mentioned above, 
the use of a scaling different from the standard one 
leads us to consider a triangular array of random variables. 
Consequently, a local limit theorem for the joint distribution of the sums 
of the positive and negative parts of the random variables 
in this triangular array plays a crucial role throughout the proof.
Equivalence of ensembles results for other particle system models 
with two conserved quantities, as well as their applications, 
have been studied in \cite{F14}.

\begin{remark}
As a further generalization of the model, 
one can consider the case where the exponent $\rho$ in the condition 
$(g)$ differs between $k \to \infty$ and $k \to -\infty$, 
namely under the following condition: 
there exist $g^+\in (0, 1]$, $g^-\in (0, 1]$, $\rho^+>-1$, $\rho^->-1$, 
such that 
\begin{align*}
g(k)= g^+\bigl(1-\frac{\rho^+}{k}+O(\frac{1}{k^2})\bigr) 
\text{ as } k\to\infty 
\text{ and }
g(k)= g^-\bigl(1-\frac{\rho^-}{k}+O(\frac{1}{k^2})\bigr) 
\text{ as } k\to -\infty. 
\end{align*}
When $g^+>g^-$, exactly the same argument as in the proof of the 
theorem below applies directly, 
showing that the limiting distribution is given by Theorem \ref{thm1} 
Case (ii) with $\rho$ replaced by $\rho^+$.
On the other hand, when $g^+\leq g^-$, 
if different exponential tiltings are applied separately to 
the positive and negative parts as described above, 
the asymptotic behaviors of the corresponding generating functions 
differ when $\rho^+\neq\rho^-$, see (\ref{q+}) below. 
Consequently, one side asymptotically dominates the other, 
and the proof of the theorem below does not work in this case.
Although numerical simulations suggest that a 
two-sided Gamma distribution whose positive-side and 
negative-side shape parameters are determined by 
$\rho^+$ and $\rho^-$, respectively, 
arises as the limiting distribution, 
we have not yet been able to prove this rigorously.
\end{remark}

In the rest of the paper 
we provide the proof of Theorem \ref{thm1} 
in Sections 2 and 3. 
We introduce the notation that will be used in the subsequent sections. 
For two sequences of positive numbers $\{a_n\}$ and $\{b_n\}$, 
$a_n\sim b_n$ means that $\lim\limits_{n\to \infty} 
\frac{a_n}{b_n}=1$, and 
$a_n \asymp b_n$ means that 
there exist constants $0<C \leq C' <\infty$ such that 
$C \leq \liminf\limits_{n\to \infty} 
\frac{a_n}{b_n} \leq 
\limsup\limits_{n\to \infty}\frac{a_n}{b_n} \leq C'$. 
$E^\mu[\, \cdot\, ]$ denotes 
the expectation with respect to the probability measure $\mu$. 
For a random variable $X$, 
$X^{+}=\max\{X, 0\}$ and 
$X^{-}=\max\{-X, 0\}$
denote the positive and negative part of $X$, respectively. 
Throughout the paper $C$, $C'$, $C''$ represent positive constants 
that do not depend on the size of the system $N$, 
but may depend on other parameters. 
These constants in various 
estimates may change from place to place in the paper. 


\section{Proof of the main theorem}
\subsection{Proof of Proposition \ref{prop1}}
\begin{proof}[Proof of Proposition \ref{prop1}]
Similar to the standard zero-range process, 
under the condition that the random walk on $\Lambda_N$ 
determined from the transition probability $\widetilde{\mathcal{P}}$ is 
irreducible, 
$\{X_n\}_{n\geq 0}$ is an irreducible and aperiodic 
Markov chain on the finite state space $\Omega_N^{K, L}$. 
All we have to do is to characterize the unique stationary distribution. 
Actually we show that $\mu_N^{K, L}$ defined by (\ref{stat}) 
satisfies the detailed balance condition: 
\begin{equation}\label{db}
\mu_{N}^{K, L}(\eta) P(X_{n+1} = \xi | X_n=\eta) 
=
\mu_{N}^{K, L}(\xi) P(X_{n+1} = \eta | X_n=\xi), 
\end{equation}
for every $\xi, \eta\in \Omega_N^{K, L}$. 
Thus, $\mu_N^{K,L}$ is a reversible distribution for 
$\{X_n\}_{n\geq0}$.
By the definition of the dynamics, it suffices to prove 
(\ref{db}) for every 
$\eta\in \Omega_{N}^{K, L}$ and 
$\xi = \eta^{x\to y}\in \Omega_N^{K, L}$, 
where $(x, y)\in \Lambda_N^*$. 
Since $\mu_N^{K,L}$ is defined as a product measure and 
$p$ is symmetric, it is enough to show that
\begin{align*}
g(\eta(x))G(\eta(x))G(\eta(y))= 
g(\xi(y))G(\xi(x))G(\xi(y))= 
g(\eta(y)+1)G(\eta(x)-1)G(\eta(y)+1). 
\end{align*}
This follows from the definition of $G$ in 
(\ref{margi}), which gives 
$g(k)\frac{G(k)}{G(k-1)}=1$ 
for every $k\in \mathbb{Z}$. 
\end{proof}

\subsection{Proof of Theorem \ref{thm1}; the case $g^+ \leq g^-$}
In this subsection we assume that $g^+ \leq g^-$. 
For the proof of Theorem \ref{thm1}, we prepare several notations. 
For $N \in \mathbb{N}$ and $K, L \in \mathbb{Z}_+$, 
we define the configuration space where both of 
the total wealth and the total debt are preserved. 
\begin{equation*}
\widetilde{\Omega}_N^{K, L}=\bigl\{
\xi = \{\xi(x)\}_{x\in \Lambda_N} \in \mathbb{Z}^{\Lambda_N}; 
\sum\limits_{x\in \Lambda_N}\xi(x) =K, 
\sum\limits_{x\in \Lambda_N} (-\xi(x))I(\xi(x) \leq -1) 
= L \bigr\}.
\end{equation*}
Then, the corresponding canonical measure is 
defined by 
\begin{equation}\label{stat2}
\widetilde{\mu}_{N}^{K, L}(\xi) = 
\frac{1}{\widetilde{Z}_{N}^{K, L}}
\bigl\{\prod_{x\in \Lambda_N}\!\! G(\xi(x))\bigr\} 
I \bigl(\xi \in \widetilde{\Omega}_N^{K, L} \bigr),\  
\xi \in \mathbb{Z}^{\Lambda_N}, 
\end{equation}
where $G$ is given by (\ref{margi}) and 
$\widetilde{Z}_{N}^{K, L}= 
\sum\limits_{\xi \in \widetilde{\Omega}_N^{K, L}}
\bigl\{\prod\limits_{x\in \Lambda_N} G(\xi(x))\bigr\}$ 
is the normalization factor. 
Since ${\Omega}_N^{K, L}$ can be expressed as the disjoint union 
$
{\Omega}_N^{K, L} = \bigcup\limits_{l=0}^L 
\widetilde{\Omega}_N^{K, l}$, 
we have the decomposition: 
\begin{align}\label{dec}
\mu_N^{K, L}(\xi)= \sum\limits_{l=0}^L h_N(l) 
\widetilde{\mu}_N^{K, l}(\xi),\  \xi \in \Omega_N^{K, L}, 
\end{align} 
where 
$
h_N(l)= \frac{\widetilde{Z}_{N}^{K, l}}
{{Z}_{N}^{K, L}}$. 
In particular, 
$\{ h_N(l)\}_{0\leq l \leq L}$ provides a probability weight 
on $\{0, 1, \cdots, L\}$. 

Next, we introduce the grand canonical measure corresponding to (\ref{stat2}). 
Though the exponential tilting (Laplace transform) 
of the probability weight is 
one of the basic tool of the analysis of the stationary distribution of 
the zero-range process (cf. \cite{KL}), there is no $\phi >0$ 
satisfying $\sum\limits_{k\in \mathbb{Z}}\phi^k G(k)<\infty$ 
when $g^+ \leq g^-$. 
Therefore, we consider exponential tiltings of the positive part and 
negative part of $G$, separately. 
For $n \in \mathbb{Z}_+$, 
we define 
$Q_n^+(\phi)=\sum\limits_{k\geq 0} k^n \phi^k G(k)
= \sum\limits_{k\geq 0} k^n \phi^k \frac{1}{g(k)!}$ 
and 
$Q_n^-(\psi)=\sum\limits_{k\leq -1} |k|^n \psi^{|k|} G(k)
= \sum\limits_{k\geq 1} k^n \psi^k {g(-k+1)!}$.  
By the condition $(g)$, 
$Q_n^+(\phi)$ converges if $\phi \in (0, g^+)$ and 
$Q_n^-(\psi)$ converges if $\psi \in (0, \frac{1}{g^-})$. 
In this case, we can define the probability measure 
$\nu_{\phi, \psi}$ on $\mathbb{Z}$ by the following. 
\begin{equation}\label{nu}
\nu_{\phi, \psi}(k)= \frac{1}{Z(\phi, \psi)}\phi^{kI(k\geq 0)}
\psi^{(-k)I(k\leq -1)}G(k),\ \ k\in \mathbb{Z}, 
\end{equation}
where $G$ is given by (\ref{margi}) and 
$Z(\phi, \psi)= Q_0^+(\phi)+Q_0^-(\psi)$ is the normalization 
factor. 
For $A \subset \mathbb{N}$, 
$\overline{\nu}_{\phi, \psi}^A$ 
be the product measure on 
$\mathbb{Z}^A$ whose one site marginal distribution equals to 
$\nu_{\phi, \psi}$. 
We denote the corresponding measure as $\overline{\nu}_{\phi, \psi}^N$ 
when $A=\Lambda_N$. 
Then, we have 
\begin{equation}\label{key1}
\overline{\nu}_{\phi, \psi}^{N}
(\widetilde{\Omega}_N^{K, L} ) 
= \frac{1}{Z(\phi, \psi)^N}\phi^{K+L}\psi^L  
\widetilde{Z}_{N}^{K, L}, 
\end{equation}
for every $K, L\in \mathbb{Z}_+$. 
In particular, we have the following key identity for (\ref{stat2}). 
\begin{equation}\label{key2}
\widetilde{\mu}_{N}^{K, L}(\, \cdot\,) = 
\overline{\nu}_{\phi, \psi}^{N}
\bigl(\, \cdot\, \bigm| \widetilde{\Omega}_N^{K, L} \bigr). 
\end{equation}
Notice that the right-hand side does not depend on the 
choice of the parameters $\phi$ and $\psi$. 

By the condition $(g)$, 
it holds that 
$\lim\limits_{\phi \uparrow g^+}Q_n^+(\phi) = \infty$ 
and 
$\lim\limits_{\psi \uparrow \frac{1}{g^-}}Q_n^-(\psi) = \infty$ 
for every $n\in \mathbb{Z}_+$. 
To investigate the behavior of 
$Q_n^+(\phi)$ and $Q_n^-(\psi)$ in more detail,
we use a Tauberian theorem of the following form 
(cf. {\cite[Corollary 1.7.3]{BGT}}). 
\begin{thm}
Let $\{a_k\}_{k\geq 0}$ be a sequence of non-negative 
numbers and assume that 
$A(s)=\sum\limits_{k=0}^\infty a_k s^k$ converges 
for $s\in [0, 1)$ and $\{a_k\}_{k\geq 0}$ is monotone. 
Then, the following are equivalent. 
\begin{itemize}
\item
$A(s) \sim \Gamma(\beta +1) (1-s)^{-\beta} h(\frac{1}{1-s})$ 
as $s \uparrow 1$ for $\beta >0$ and slowly varying function $h$. 
\item
$a_k \sim \beta k^{\beta-1} h(k)$ as $k\to \infty$ 
for $\beta >0$ and slowly varying function $h$. 
\end{itemize}
\end{thm}
\noindent
By this theorem and (\ref{const}), 
the following asymptotics holds. 
\begin{align}\label{q+}
\begin{split}
Q_n^+(\phi) 
& \sim 
C^+ \Gamma(\rho + n+1)
\bigl(1- \frac{\phi}{g^+}\bigr)^{-(\rho +n+1)}\  
\text{ as } \phi \uparrow g^+, \\
Q_n^-(\psi) 
& \sim 
C^- \Gamma(\rho + n+1)
\bigl(1- g^- {\psi}\bigr)^{-(\rho +n+1)}\  
\text{ as } \psi \uparrow \frac{1}{g^-}. 
\end{split}
\end{align}

Next, 
we show that, given the values of $E^{\nu_{\phi, \psi}}[\xi^+]$ and 
$E^{\nu_{\phi, \psi}}[\xi^-]$, 
there exists a unique pair $(\phi, \psi)$ realizing them.
\begin{lem}\label{lem1}
Let $\kappa^+, \kappa^->0$ and 
$\{a_N\}_{N \geq 1}$ be an increasing 
sequence of positive numbers 
which satisfies $\lim\limits_{N\to \infty}a_N=\infty$. 
There exists a unique pair of 
two increasing sequences of positive numbers 
$\{\phi_N\}_{N \geq 1}$ and $\{\psi_N\}_{N \geq 1}$ 
which satisfy $E^{\nu_{\phi_N, \psi_N}}[\xi^+] = \kappa^+ a_N$ and 
$E^{\nu_{\phi_N, \psi_N}}[\xi^-] = \kappa^- a_N$ 
for every $N\geq 1$. 
Moreover, $\{\phi_N\}_{N \geq 1}$ and $\{\psi_N\}_{N \geq 1}$ 
satisfy 
$1-\frac{\phi_N}{g^+} \sim \frac{1}{\theta^+ a_N}$, 
$1-g^- {\psi_N} \sim \frac{1}{\theta^- a_N}$ as $N\to \infty$ 
where 
$(\theta^+, \theta^-)$ is 
the unique solution of the following equation 
in $(0, \infty)\times (0, \infty)$. 
\begin{eqnarray}\label{th1}
\begin{cases}
& \frac{(\rho+1) C^+(\theta^+)^{\rho + 2}}
{C^+(\theta^+)^{\rho + 1}+ C^-(\theta^-)^{\rho + 1}}
= \kappa^+, \\
& \frac{(\rho+1) C^-(\theta^-)^{\rho + 2}}
{C^+(\theta^+)^{\rho + 1}+ C^-(\theta^-)^{\rho + 1}}
= \kappa^-.
\end{cases}
\end{eqnarray}
\end{lem}

\begin{proof}
Set $F_1(\phi, \psi) := \frac{Q_1^+(\phi)}
{Q_0^+(\phi) + Q_0^-(\psi)}, 
F_2(\phi, \psi) := \frac{Q_1^-(\psi)}
{Q_0^+(\phi) + Q_0^-(\psi)}$ and 
consider a map 
$F=(F_1, F_2):(0, g^+)\times (0, \frac{1}{g^-}) \to 
(0, \infty)\times (0, \infty)$. 
Then, it is easy to see that the following holds. 
\begin{itemize}
\item
$\lim\limits_{\phi \downarrow 0}F_1(\phi, \psi)=0$ for every 
$\psi \in (0, \frac{1}{g^-})$ and 
$\lim\limits_{\psi \downarrow 0}F_2(\phi, \psi)=0$ for every 
$\phi \in (0, {g^+})$. 
\item
$\lim\limits_{\phi \uparrow g^+}F_1(\phi, \psi)=\infty$ for every 
$\psi \in (0, \frac{1}{g^-})$ and 
$\lim\limits_{\psi \uparrow \frac{1}{g^-}}F_2(\phi, \psi)=
\infty$ for every $\phi \in (0, {g^+})$. 
\item
$F$ is $C^1$ and 
$\frac{\partial F_1}{\partial \phi}>0$, 
$\frac{\partial F_2}{\partial \psi}>0$, 
$\mathrm{det}DF=
\frac{\partial F_1}{\partial \phi}\frac{\partial F_2}{\partial \psi}
-\frac{\partial F_1}{\partial \psi}\frac{\partial F_2}{\partial \phi}>0$ 
for every $(\phi, \psi)$. 
\end{itemize}
Therefore $F$ is a bijection and there exists a 
unique pair 
$(\{\phi_N\}_{N \geq 1}, \{\psi_N\}_{N \geq 1})$ 
which satisfies $E^{\nu_{\phi_N, \psi_N}}[\xi^+]  
= F_1(\phi_N, \psi_N) =\kappa^+ a_N$ 
and $E^{\nu_{\phi_N, \psi_N}}[\xi^-] = 
F_2(\phi_N, \psi_N) = \kappa^- a_N$ 
for every $N\geq 1$. 
Then, we have $Q_1^+(\phi_N) \asymp Q_1^-(\psi_N)$ 
as $N\to \infty$ 
since $F_1(\phi_N, \psi_N)$ and 
$F_2(\phi_N, \psi_N)$ are of the same order. 
Combining these observations with (\ref{q+}), 
we obtain 
$1-\frac{\phi_N}{g^+} \sim \frac{1}{\theta^+ a_N}$ 
and $1-g^- {\psi_N} \sim \frac{1}{\theta^- a_N}$ 
as $N\to \infty$ for some $\theta^+, \theta^->0$. 
Substituting these relations back into equations 
$F_1(\phi_N, \psi_N) =\kappa^+ a_N$, 
$F_2(\phi_N, \psi_N) = \kappa^- a_N$ 
and applying (\ref{q+}) again,
we see that $(\theta^+, \theta^-)$ satisfies equation 
(\ref{th1}) which 
can be solved explicitly, 
and the uniqueness of the solution follows as well.

\end{proof}

Based on Lemma \ref{lem1}, we consider 
sequences $\{\phi_N\}_{N\geq 1}$ and 
$\{\psi_N\}_{N\geq 1}$ that satisfy 
$1-\frac{\phi_N}{g^+} \sim \frac{1}{\lambda^+ a_N}$
and 
$1-g^-{\psi_N} \sim \frac{1}{\lambda^- a_N}$ 
as $N\to \infty$ 
where 
$\lambda^+, \lambda^- > 0$ and 
$\{a_N\}_{N\geq 1}$ be an increasing 
sequence of positive numbers which satisfies 
$\lim\limits_{N\to \infty}a_N = \infty$. 
For each $N\in \mathbb{N}$, let 
$\{X_j^{(N)}\}_{j\in \Lambda_N}$ be a family of 
independent and identically distributed 
$\mathbb{Z}$-valued random variables 
with common distribution $\nu_{\phi_N, \psi_N}$ given by 
(\ref{nu}). 
In the following, 
we denote 
$X_j^{(N)}$, $(X_j^{(N)})^+$ and $(X_j^{(N)})^-$ 
as 
$X_j$, $X_j^{+}$ and $X_j^{-}$ 
for notational simplicity. 
We set 
$m_N^+ = E[X_1^{+}]$, 
$m_N^- = E[X_1^{-}]$, 
$v_N^+ = \mathrm{Var}(X_1^{+})$, 
$v_N^- = \mathrm{Var}(X_1^{-})$. 
Then, the following asymptotics holds by (\ref{q+}). 
\begin{align}\label{mv}
\begin{split}
Z(\phi_N, \psi_N) & = 
Q_0^+(\phi_N) + Q_0^-(\psi_N) 
= (B_0^+(\lambda^+)+B_0^-(\lambda^-))(a_N)^{\rho+1}(1+o(1)), \\
m_N^+& = m_N^+(\lambda^+, \lambda^-) 
= \frac{Q_1^+(\phi_N)}{Z(\phi_N, \psi_N)} 
= \frac{B_1^+(\lambda^+)}{B_0^+(\lambda^+)+B_0^-(\lambda^-)}
a_N(1+o(1)), \\
m_N^- & = m_N^-(\lambda^+, \lambda^-) 
= \frac{{Q}_1^-(\psi_N)}{Z(\phi_N, \psi_N)}
= \frac{B_1^-(\lambda^-)}
{B_0^+(\lambda^+)+B_0^-(\lambda^-)} a_N(1+o(1)), \\
v_N^+ & = v_N^+(\lambda^+, \lambda^-) 
= \frac{Q_2^+(\phi_N)}{{Z}(\phi_N, \psi_N)}
-\Bigl(\frac{Q_1^+(\phi_N)}{{Z}(\phi_N, \psi_N)}\Bigr)^2 \\
& = \frac{B_2^+(\lambda^+)(B_0^+(\lambda^+) + B_0^-(\lambda^-))
-(B_1^+(\lambda^+))^2}{(B_0^+(\lambda^+) + B_0^-(\lambda^-))^2}
 (a_N)^2(1+o(1)), \\
v_N^- & = v_N^+(\lambda^+, \lambda^-) 
= \frac{Q_2^-(\psi_N)}{{Z}(\phi_N, \psi_N)}
-\Bigl(\frac{Q_1^-(\psi_N)}{{Z}(\phi_N, \psi_N)}\Bigr)^2 \\
& = \frac{B_2^-(\lambda^-)(B_0^+(\lambda^+) 
+ B_0^-(\lambda^-))-(B_1^-(\lambda^-))^2}
{(B_0^+(\lambda^+) + B_0^-(\lambda^-))^2}
 (a_N)^2(1+o(1)), 
\end{split}
\end{align}
as $N\to \infty$ 
where 
$B_n^+(\lambda^+) = C^+ \Gamma(\rho + n+1)(\lambda^+)^{\rho + n +1}$ 
and 
$B_n^-(\lambda^-) = C^-\Gamma(\rho + n+1)(\lambda^-)^{\rho + n +1}$, 
$n\in \mathbb{Z}_+$.   
In particular,  we have 
\begin{align}\label{tau}
\tau := \tau(\lambda^+, \lambda^-) =
\lim\limits_{N\to \infty}
\frac{\mathrm{Cov}(X_1^{+}, X_1^{-})}
{\sqrt{\mathrm{Var}(X_1^{+})
\mathrm{Var}(X_1^{-})}}
= -
\lim\limits_{N\to \infty}
\frac{m_N^{+} m_N^{-}}
{\sqrt{v_N^{+}v_N^{-}}} \in (-1, 0).
\end{align} 
Then, we have the following local limit theorem 
for the joint distribution of 
$\sum\limits_{j \in \Lambda_N} X_j^{ +}$ and 
$\sum\limits_{j \in \Lambda_N} X_j^{ -}$. 
We provide the proof in Section 3. 
\begin{thm}\label{lclt}
Assume that 
$a_n=o(n^\beta)$ for some $\beta>0$. 
Then, for every finite set $B\subset \mathbb{N}$, it holds that 
\begin{align*}
\lim\limits_{N\to \infty} 
\sup\limits_{(K, L)\in \mathbb{Z}_+\times \mathbb{Z}_+} 
\Bigm|\! (N-|B|)
\sqrt{v_N^+ v_N^-} 
P \Bigl(
\sum\limits_{j \in \Lambda_N \setminus B} X_j^{ +} =K, 
\sum\limits_{j \in \Lambda_N \setminus B} X_j^{ -} =L 
\Bigr) -
q_0(\overline{K}_N, \overline{L}_N)
\!\Bigm| =0, 
\end{align*}
where
$ 
q_0(x, y)= 
\frac{1}{2\pi \sqrt{1-\tau^2}}
\exp\Bigl\{-\frac{x^2 -2\tau xy+y^2}{2(1-\tau^2)}
\Bigr\}$, 
$\overline{K}_N= \frac{K-(N-|B|)m_N^+}{\sqrt{(N-|B|)v_N^+}}$
and 
$\overline{L}_N= \frac{L-(N-|B|)m_N^-}{\sqrt{(N-|B|)v_N^-}}$. 
\end{thm}
\noindent

In what follows, 
for $A\subset\Lambda_N$ and 
$\eta=\{\eta(x)\}_{x\in A} \in \mathbb{Z}^A$, 
we define
$S_A(\eta)=\sum\limits_{x\in A}\eta(x)$, 
$S_A^+(\eta)=\sum\limits_{x\in A} \eta(x)^+$ 
and $S_A^-(\eta)=\sum\limits_{x\in A}\eta(x)^-$. 
In particular, when we simply write $S_A$, $S_A^+$, or $S_A^-$, 
they are understood as the corresponding random variables on 
$\mathbb{Z}^A$. Also, when $A=\Lambda_N$, 
we simply write these as $S_N$, $S_N^+$ and $S_N^-$, respectively.
With the above preparations, 
we first show that the total debt concentrates around $L_N$ 
under the measure $\mu_N^{K_N, L_N}$ when $g^+ \leq g^-$. 

\begin{lem}\label{lem2}
Assume that $g^+ \leq g^-$. 
Set $K_N = \kappa a_N N$ and 
$L_N = r \kappa a_N N$ where 
$\kappa, r>0$ and 
$\{a_N\}_{N \geq 1}$ be an increasing 
sequence of positive numbers which satisfies 
$\lim\limits_{N\to \infty}a_N=\infty$ and 
$a_N=o(N^\beta)$ for some $\beta>0$. 
Then, 
for every sequence of positive numbers 
$\{t_N\}_{N\geq 1}$ which satisfies 
$\lim\limits_{N\to \infty} \frac{t_N}{a_N N^{\beta'}}= 
\infty$ for some $\beta'>0$, 
we have 
$\lim\limits_{N\to \infty} 
\mu_{N}^{K_N, L_N}(S_N^- \leq L_N-t_N)=0$. 
\end{lem}
\begin{proof}
Set 
$h_N(l)= \frac{\widetilde{Z}_{N}^{K_N, l}}{{Z}_{N}^{K_N, L_N}} 
= \mu_{N}^{K_N, L_N}(S_N^- =l )$, 
$0 \leq l \leq L_N$. 
We have 
\begin{align*}
\mu_{N}^{K_N, L_N}(S_N^- \leq L_N-t_N)
= 
\sum\limits_{l=0}^{L_N-t_N}h_{N}(l) 
& \leq
\sum\limits_{l=0}^{L_N-t_N}
\frac{\widetilde{Z}_N^{K_N, l}}{\widetilde{Z}_N^{K_N, L_N}} \\
& = 
\sum\limits_{l=0}^{L_N-t_N} \phi^{L_N-l}\psi^{L_N-l}
\frac{\overline{\nu}^N_{\phi, \psi}(\widetilde{\Omega}_N^{K_N, l})}
{\overline{\nu}^N_{\phi, \psi}(\widetilde{\Omega}_N^{K_N, L_N})} 
\leq 
\frac{L_N \bigl(\max\limits_{t_N \leq l \leq L_N}\{\phi^l \psi^l\}\bigr)}
{\overline{\nu}^N_{\phi, \psi}(\widetilde{\Omega}_N^{K_N, L_N})}, 
\end{align*}
where we used (\ref{key1}) for the second equality. 
Notice that this estimate holds for every 
$\phi$ and $\psi$ which satisfy 
$\phi \in (0, g^+)$ and $\psi \in (0, \frac{1}{g^-})$. 
We therefore take the pair $(\phi, \psi)$ as 
$(\phi^*_N, \psi^*_N)$ such that 
$E^{\nu_{\phi^*_N, \psi^*_N}}[\xi^+ ] = (1+r)\kappa a_N$ 
and 
$E^{\nu_{\phi^*_N, \psi^*_N}}[\xi^- ] = r\kappa a_N$, 
which is possible by Lemma \ref{lem1}. 

By Lemma \ref{lem1} we know that 
$1-\frac{\phi^*_N}{g^+} \sim \frac{1}{\theta^+ a_N}$ and  
$1-g^- {\psi^*_N} \sim \frac{1}{\theta^- a_N}$ 
as $N\to \infty$ where 
$(\theta^+, \theta^-)$ is a solution of (\ref{th1}) 
with $\kappa^+ = (1+r)\kappa$ and 
$\kappa^- = r \kappa$. 
Then, by applying Theorem \ref{lclt}, we have 
\begin{align*}
{\overline{\nu}^N_{\phi^*_N, \psi^*_N}(\widetilde{\Omega}_N^{K_N, L_N})}
& = 
{\overline{\nu}^N_{\phi^*_N, \psi^*_N}}\bigl(
\sum\limits_{x\in \Lambda_N}\xi(x)^+ = (1+r)\kappa a_N N, 
\sum\limits_{x\in \Lambda_N}\xi(x)^- = r \kappa a_N N
\bigr) \\
& = \frac{1}{2\pi \sqrt{1-\tau^2} N \sqrt{v_N^+ v_N^-}}
(1+o(1)), 
\end{align*}
as $N\to \infty$ where 
$v_N^+ = v_N^+( \theta^+, \theta^-)$, 
$v_N^- = v_N^-( \theta^+, \theta^-)$ and 
$\tau = \tau( \theta^+, \theta^-)$ 
are given by (\ref{mv}) and (\ref{tau}). 
To summarize, 
\begin{align*}
\mu_{N}^{K_N, L_N}(S_N^- \leq L-t_N)
& \leq 
C N \sqrt{v_N^+ v_N^-} \cdot L_N \Bigl\{
g^+ \bigl( 1-\frac{1+o(1)}{\theta^+ a_N}\bigr) 
\frac{1}{g^-} \bigl( 1-\frac{1+o(1)}{\theta^- a_N}\bigr)
\Bigr\}^{t_N} \\
& \leq 
C'N^2 (a_N)^3 e^{-\frac{C''t_N}{aN}}, 
\end{align*}
where we used the condition $g^+ \leq g^-$ 
for the last inequality. 
The right hand goes to $0$ as $N\to \infty$ by the assumption 
on $a_N$ and $t_N$. 
\end{proof}

Now we are in the position to give the proof of Theorem \ref{thm1}. 
\begin{proof}[Proof of Theorem \ref{thm1}; the case $g^+ \leq g^-$]
Let $f:\mathbb{R}^A \to \mathbb{R}$ be a 
bounded continuous local function 
where $A$ is a finite subset of $\mathbb{N}$. 
For every $N\in \mathbb{N}$ so that $\Lambda_N \supset A$, 
by (\ref{dec}) and Fubini's theorem, we have 
\begin{align*}
\begin{split}\label{dec}
& E^{\mu_{N}^{K_N, L_N}}\bigl[ f\bigl( 
\frac{\cdot}{a_N} \bigr)\bigr]  \\
& \quad = 
\sum\limits_{\eta\in \mathbb{Z}^A} 
f\bigl( \frac{\eta}{a_N} \bigr) 
{\mu_{N}^{K_N, L_N}}( \xi_A = \eta) \\
& \quad = 
\sum\limits_{\eta\in \mathbb{Z}^A} 
f\bigl( \frac{\eta}{a_N} \bigr) 
\Bigl\{
\sum\limits_{l=0}^{L_N}h_N(l)
{\widetilde{\mu}_{N}^{K_N, l}}( \xi_A = \eta) \Bigr\}\\
& \quad = 
\sum\limits_{l=0}^{L_N-t_N}h_N(l)
\Bigl\{\sum\limits_{\eta\in \mathbb{Z}^A} 
f\bigl( \frac{\eta}{a_N} \bigr) 
{\widetilde{\mu}_{N}^{K_N, l}}( \xi_A = \eta) \Bigr\}
+ 
\sum\limits_{l=L_N-t_N+1}^{L_N}h_N(l)
\Bigl\{
\sum\limits_{\eta\in \mathbb{Z}^A} 
f\bigl( \frac{\eta}{a_N} \bigr) 
{\widetilde{\mu}_{N}^{K_N, l}}( \xi_A = \eta) \Bigr\}\\
& \quad =: I_1 +I_2,
\end{split} 
\end{align*}
where $\xi_A$ denotes the configurations 
$\xi$ restricted to the set $A$ and 
$\{t_N\}_{N\geq 1}$ be a sequence of positive numbers 
which satisfies 
$\lim\limits_{N\to \infty} \frac{t_N}{a_N N^\beta}= 
\infty$ for some $\beta >0$ and 
$\lim\limits_{N\to \infty} \frac{t_N}{a_N N^{\frac{1}{2}}}= 0$. 

For $I_1$, we have 
\begin{align*}
|I_1| \leq \| f \|_{\infty} 
\sum\limits_{l=0}^{L_N-t_N}h_N(l)
= \| f \|_{\infty} 
\mu_N^{K_N, L_N}(S_N^- \leq L_N -t_N) 
\to 0,
\end{align*}
as $N\to \infty$ by Lemma \ref{lem2}. 
Next, for each $L_N-t_N+1 \leq l \leq L_N$, 
\begin{align*}
& \sum\limits_{\eta\in \mathbb{Z}^A} 
f\bigl( \frac{\eta}{a_N} \bigr) 
{\widetilde{\mu}_{N}^{K_N, l}}( \xi_A = \eta) \\
& = 
\sum\limits_{\eta\in \mathbb{Z}^A} 
f\bigl( \frac{\eta}{a_N} \bigr) 
\overline{\nu}_{\phi, \psi}^{\Lambda_N}
\bigl( \xi_A = \eta \bigm| \widetilde{\Omega}_N^{K_N, l} \bigr)\\
& = 
\sum\limits_{\eta\in \mathbb{Z}^A} 
f\bigl( \frac{\eta}{a_N} \bigr) 
\frac{
\overline{\nu}_{\phi, \psi}^{A}(\eta)
\overline{\nu}_{\phi, \psi}^{\Lambda_N \setminus A}
\bigl( S_{\Lambda_N \setminus A} = K_N-S_A(\eta), 
S_{\Lambda_N \setminus A}^- = l-S_A^-(\eta)\bigr)}
{\overline{\nu}_{\phi, \psi}^{N}
(\widetilde{\Omega}_N^{K_N, l} \bigr)}\\
& = 
\sum\limits_{\eta\in \mathbb{Z}^A} 
f\bigl( \frac{\eta}{a_N} \bigr) 
\overline{\nu}_{\phi, \psi}^{A}(\eta) 
+ 
\sum\limits_{\eta\in \mathbb{Z}^A} 
f\bigl( \frac{\eta}{a_N} \bigr) 
\overline{\nu}_{\phi, \psi}^{A}(\eta)
\Bigl\{
\frac{
\overline{\nu}_{\phi, \psi}^{\Lambda_N \setminus A}
\bigl( S_{\Lambda_N \setminus A} = K_N-S_A(\eta), 
S_{\Lambda_N \setminus A}^- = l-S_A^-(\eta)\bigr)}
{\overline{\nu}_{\phi, \psi}^{N}
(\widetilde{\Omega}_N^{K_N, l} \bigr)} -1 \Bigr\}. 
\end{align*}
The first equality follows from (\ref{key2}). 
In particular, the right hand side does not depend on the 
choice of $\phi$ and $\psi$. 
Then, 
\begin{align*}
I_2 & = 
\Bigl\{\sum\limits_{l=L_N-t_N+1}^{L_N}h_N(l)\Bigr\}
\Bigl\{\sum\limits_{\eta\in \mathbb{Z}^A} 
f\bigl( \frac{\eta}{a_N} \bigr) 
\overline{\nu}_{\phi, \psi}^{A}(\eta) \Bigr\} \\
& \qquad 
+ 
\sum\limits_{l=L_N-t_N+1}^{L_N}h_N(l)
\Bigl[
\sum\limits_{\eta\in \mathbb{Z}^A} 
f\bigl( \frac{\eta}{a_N} \bigr) 
\overline{\nu}_{\phi, \psi}^{A}(\eta) 
\Bigl\{
\frac{
\overline{\nu}_{\phi, \psi}^{\Lambda_N \setminus A}
\bigl( S_{\Lambda_N \setminus A} = K_N-S_A(\eta), 
S_{\Lambda_N \setminus A}^- = l-S_A^-(\eta)\bigr)}
{\overline{\nu}_{\phi, \psi}^{N}
(\widetilde{\Omega}_N^{K_N, l} \bigr)} -1 
\Bigr\}\Bigr]\\
& =: I_3 +I_4. 
\end{align*}
By using Lemma \ref{lem2} again, 
\begin{align*}
I_3 = 
\mu_N^{K_N, L_N}(S_N^- \geq L_N -t_N+1) 
\Bigl\{\sum\limits_{\eta\in \mathbb{Z}^A} 
f\bigl( \frac{\eta}{a_N} \bigr) 
\overline{\nu}_{\phi, \psi}^{A}(\eta) \Bigr\} 
= 
\Bigl\{\sum\limits_{\eta\in \mathbb{Z}^A} 
f\bigl( \frac{\eta}{a_N} \bigr) 
\overline{\nu}_{\phi, \psi}^{A}(\eta) \Bigr\} 
(1+o(1)), 
\end{align*}
as $N\to \infty$. 
For $I_4$, 
\begin{align*}
|I_4| 
& \leq 
\|f \|_{\infty} 
\max\limits_{L_N-t_N +1 \leq l \leq L_N}
\Bigl\{
\sum\limits_{\eta\in \mathbb{Z}^A} 
\overline{\nu}_{\phi, \psi}^{A}(\eta) 
\Bigm|
\frac{
\overline{\nu}_{\phi, \psi}^{\Lambda_N \setminus A}
\bigl( S_{\Lambda_N \setminus A} = K_N-S_A(\eta), 
S_{\Lambda_N \setminus A}^- = l-S_A^-(\eta)\bigr)}
{\overline{\nu}_{\phi, \psi}^{N}
(\widetilde{\Omega}_N^{K_N, l} \bigr)} -1 
\Bigm| \Bigr\}. 
\end{align*}
Therefore, all we have to show is 
\begin{equation}\label{main}
\sum\limits_{\eta\in \mathbb{Z}^A} 
f\bigl( \frac{\eta}{a_N} \bigr) 
\overline{\nu}_{\phi, \psi}^{A}(\eta) 
\to E^{\overline{\mu}_{\kappa, r}^A}[f], 
\end{equation}
and 
\begin{equation*}
I_5:= 
\max\limits_{L_N-t_N +1 \leq l \leq L_N}
\Bigl\{
\sum\limits_{\eta\in \mathbb{Z}^A} 
\overline{\nu}_{\phi, \psi}^{A}(\eta)
\Bigm|
\frac{
\overline{\nu}_{\phi, \psi}^{\Lambda_N \setminus A}
\bigl( S_{\Lambda_N \setminus A} = K_N-S_A(\eta), 
S_{\Lambda_N \setminus A}^- = l-S_A^-(\eta)\bigr)}
{\overline{\nu}_{\phi, \psi}^{N}
(\widetilde{\Omega}_N^{K_N, l} \bigr)} -1 
\Bigm| \Bigr\} 
\to 0, 
\end{equation*}
as $N\to \infty$ 
where $\mu_{\kappa, r}$ denotes the corresponding 
limiting distribution in Theorem \ref{thm1}.
For this purpose, we again choose $\phi^*_N$ and $\psi^*_N$ 
so that 
$E^{\nu_{\phi^*_N, \psi^*_N}}[\xi^+ ] = (1+r)\kappa  a_N$ 
and 
$E^{\nu_{\phi^*_N, \psi^*_N}}[\xi^- ] = r \kappa a_N$ 
hold. 

For the proof of the convergence of the main term (\ref{main}), 
we assume that 
$f$ is a function of one variable for notational simplicity. 
The general case can be proved by the similar manner since 
$\overline{\nu}_{\phi, \psi}^A$ 
is a product measure with the same marginal distribution. 
For $R>0$, we have 
\begin{align*}
& \sum\limits_{\eta\in \mathbb{Z}} 
f\bigl( \frac{\eta}{a_N} \bigr) 
{\nu}_{\phi^*_N, \psi^*_N}(\eta) \\
& = 
\sum\limits_{k =0}^{Ra_N} 
f\bigl( \frac{k}{a_N} \bigr) \frac{1}{Z({\phi^*_N, \psi^*_N})}(\phi^*_N)^k 
\frac{1}{g(k)!}
+ 
\sum\limits_{k= -Ra_N}^{-1} 
f\bigl( \frac{k}{a_N} \bigr) \frac{1}{Z({\phi^*_N, \psi^*_N})}(\psi^*_N)^{|k|} 
{g(k+1)!} \\
& \qquad \qquad + 
\sum\limits_{|k| \geq Ra_N} 
f\bigl( \frac{k}{a_N} \bigr) 
{\nu}_{\phi^*_N, \psi^*_N}(k) \\
& =: I_6+ I_7 +I_8. 
\end{align*}
By (\ref{const}), (\ref{mv}) and Lemma \ref{lem1},  
\begin{align*}
I_6 
& = 
\frac{1+o(1)}{\{C^+ \Gamma(\rho +1) (\theta^+)^{\rho +1}
+ C^- \Gamma(\rho +1) (\theta^-)^{\rho +1}\}(a_N)^{\rho+1}}
\sum\limits_{k =0}^{Ra_N} 
f\bigl( \frac{k}{a_N} \bigr) 
\Bigl\{g^+\bigl(1-\frac{1}{\theta^+ a_N}\bigr)\Bigr\}^k 
\frac{C^+ k^{\rho}}{(g^+)^k} + o(1) \\
& = 
\frac{C^+(1+o(1))}{\Gamma(\rho +1)\{C^+ (\theta^+)^{\rho +1}
+ C^- (\theta^-)^{\rho +1}\}}
\frac{1}{a_N}
\sum\limits_{k =0}^{Ra_N} 
f\bigl( \frac{k}{a_N} \bigr) 
e^{-\frac{1}{\theta^+}\frac{k}{a_N}} 
\bigl(\frac{k}{a_N}\bigr)^\rho + o(1) \\
& \to 
\frac{C^+}{\Gamma(\rho +1)\{C^+ (\theta^+)^{\rho +1}
+ C^- (\theta^-)^{\rho +1}\}}
\int_0^R 
f(t) e^{-\frac{1}{\theta^+}t} 
t^\rho dt, 
\end{align*}
as $N\to \infty$ where the convergence 
follows from the definition of the Riemann integral. 
We note that 
since $Z({\phi^*_N, \psi^*_N}) \to \infty$ in the limit $N\to \infty$, 
a finite sum in $I_6$ is negligible and 
by the condition $(g)$, 
$\frac{1}{g(k)!}$ can be replaced with 
$\frac{C^+ k^{\rho}}{(g^+)^k}$
up to an error of $o(1)$.
By taking the limit $R\to \infty$, 
the integral in right-hand side converges to 
$\int_0^\infty
f(t) e^{-\frac{1}{\theta^+}t} 
t^\rho dt$. 
Similarly, 
\begin{align*}
I_7 
& = 
\frac{1+o(1)}{\{C^+ \Gamma(\rho +1) (\theta^+)^{\rho +1}
+ C^- \Gamma(\rho +1) (\theta^-)^{\rho +1}\}(a_N)^{\rho+1}}
\sum\limits_{k =1}^{Ra_N} 
f\bigl( \frac{-k}{a_N} \bigr) 
\Bigl\{\frac{1}{g^-}\bigl(1-\frac{1}{\theta^- a_N}\bigr)\Bigr\}^k 
{(g^-)^k C^- k^{\rho}} + o(1) \\
& \to 
\frac{C^-}{\Gamma(\rho +1)\{C^+ (\theta^+)^{\rho +1}
+ C^- (\theta^-)^{\rho +1}\}}
\int_0^\infty
f(-t) e^{-\frac{1}{\theta^-}t} 
t^\rho dt, 
\end{align*}
as $N\to \infty$ and $R \to \infty$. 
For $I_8$, we have 
\begin{align*}
\limsup\limits_{N\to \infty}|I_8| 
& \leq C \| f\|_{\infty} \Bigl\{
\int_R^\infty 
e^{-\frac{1}{\theta^+}t} 
t^\rho dt + 
\int_R^\infty
e^{-\frac{1}{\theta^-}t} 
t^\rho dt\Bigr\}. 
\end{align*}
The right hand side goes to $0$ as $R\to \infty$. 
As a result, we obtain the following.
\begin{align*}
& \lim\limits_{N\to \infty}
\sum\limits_{\eta\in \mathbb{Z}} 
f\bigl( \frac{\eta}{a_N} \bigr) 
{\nu}_{\phi^*_N, \psi^*_N}(\eta) \\
& \quad = 
\frac{1}{\Gamma(\rho +1)\{C^+ (\theta^+)^{\rho +1}
+ C^- (\theta^-)^{\rho +1}\}}
\Bigl\{ C^+ \int_0^\infty
f(t) e^{-\frac{1}{\theta^+}t} 
t^\rho dt +
C^- \int^0_{-\infty}
f(t) e^{-\frac{1}{\theta^-}|t|} 
|t|^\rho dt\Bigr\}. 
\end{align*}

Finally, we consider the estimate of $I_5$. 
By Theorem \ref{lclt}, we have 
\begin{align*}
\overline{\nu}_{\phi^*_N, \psi^*_N}^{N}
\bigl(\widetilde{\Omega}_N^{K_N, l} \bigr) 
& = 
\overline{\nu}_{\phi^*_N, \psi^*_N}^{N}
\bigl(S_N^{+}=K_N + l, S_N^-= l \bigr) \\
& = \frac{1}{N\sqrt{v_N^+ v_N^-}} \Bigl[ \frac{1}{2\pi\sqrt{1-\tau^2}}
\exp\Bigl\{-\frac{1}{2(1-\tau^2)} 
\Bigl(
\bigl( \frac{K_N + l- N (1+r)\kappa a_N}{\sqrt{N v_N^+}} \bigr)^2 \\
& \qquad \qquad \qquad 
-2\tau \bigl( \frac{K_N + l- N (1+r)\kappa a_N}{\sqrt{N v_N^+}} \bigr)
\bigl( \frac{ l- N r \kappa a_N}{\sqrt{N v_N^-}} \bigr) 
+ \bigl( \frac{ l- N r \kappa a_N}{\sqrt{N v_N^-}} \bigr)^2
\Bigr)\Bigr\} + o(1) \Bigr], 
\end{align*}
and 
\begin{align}\label{b2}
\begin{split}
& \overline{\nu}_{\phi^*_N, \psi^*_N}^{\Lambda_N \setminus A}
\bigl( S_{\Lambda_N\setminus A} = K_N -S_A(\eta), 
S_{\Lambda_N\setminus A}^- = l -S_A^-(\eta) \bigr) \\
& \qquad = 
\frac{1}{(N-|A|)\sqrt{v_N^+ v_N^-}} 
\Bigl[ \frac{1}{2\pi\sqrt{1-\tau^2}}
\exp\Bigl\{- \frac{1}{2(1-\tau^2)} 
\Bigl(
\bigl( \frac{K_N + l- S_A^+(\eta)-(N-|A|) (1+r) \kappa a_N}
{\sqrt{(N-|A|) v_N^+}} \bigr)^2 \\
& \qquad \qquad \qquad 
-2\tau \bigl( \frac{K_N + l- S_A^+(\eta)-(N-|A|) (1+r)\kappa a_N}
{\sqrt{(N-|A|) v_N^+}} \bigr)
\bigl( \frac{ l- S_A^-(\eta)-(N-|A|) r \kappa a_N}
{\sqrt{(N-|A|) v_N^-}} \bigr) 
\\
& \qquad \qquad \qquad \qquad \qquad 
+ \bigl( \frac{ l- S_A^-(\eta)-(N-|A|) r \kappa a_N}
{\sqrt{(N-|A|) v_N^-}} \bigr)^2
\Bigr)\Bigr\} + o(1) \Bigr]. 
\end{split}
\end{align}
Note that the $o(1)$ terms do not depend on $l$ and $\eta$. 
For every $L_N-t_N +1 \leq l \leq L_N$, we have 
$( \frac{K_N + l- N (1+r)\kappa a_N}{\sqrt{N v_N^+}} )^2 
\leq \frac{(t_N)^2}{N v_N^+} $ 
and 
$( \frac{ l- N r \kappa a_N}{\sqrt{N v_N^-}} )^2 
\leq \frac{(t_N)^2}{N v_N^-} $. 
By the choice of $t_N$ and the fact that 
$v_N^+ \asymp v_N^- \asymp a_N^2$, 
these go to $0$ as $N\to \infty$.
As a result, 
\begin{align}\label{b3}
\overline{\nu}_{\phi^*_N, \psi^*_N}^{N}
(\widetilde{\Omega}_N^{K_N, l} ) = 
\frac{1 + o(1)}{{2\pi\sqrt{1-\tau^2}}N\sqrt{v_N^+ v_N^-}}, 
\end{align}
as $N\to \infty$ uniformly in 
$L_N-t_N +1 \leq l \leq L_N$. 
Combining this with (\ref{b2}), we obtain the 
following bound for some $C_0 < \infty$. 
\begin{align*}
\sup_{N\geq 1}
\sup\limits_{ L_N-t_N +1 \leq l \leq L_N}
\sup\limits_{\eta \in \mathbb{Z}^A} 
\Bigm|
\frac{
\overline{\nu}_{\phi^*_N, \psi^*_N}^{\Lambda_N \setminus A}
\bigl( S_{\Lambda_N \setminus A} = K_N-S_A(\eta), 
S_{\Lambda_N \setminus A}^- = l-S_A^-(\eta)\bigr)}
{\overline{\nu}_{\phi^*_N, \psi^*_N}^{N}
(\widetilde{\Omega}_N^{K_N, l} \bigr)} -1 
\Bigm|\,  
\leq C_0.  
\end{align*}
Next, we consider the estimate of (\ref{b2}). 
Set $D_A(k) = \{-k, -k+1, \cdots, k\}^A$, $k\in \mathbb{Z}_+$. 
If $\eta \in D_A(\lfloor Ra_N \rfloor)$, $R>0$ then 
$|S_A(\eta)| \leq |A|Ra_N$, $0 \leq S_A^-(\eta) \leq 
|A|Ra_N$ and 
these terms are negligible compared to 
$\sqrt{N v_N^+}$, $\sqrt{N v_N^-}$. 
Therefore, by the same argument as above 
$\bigl( \frac{K_N + l- S_A^+(\eta)-(N-|A|) (1+r)\kappa a_N}
{\sqrt{(N-|A|) v_N^+}} \bigr)^2 =o(1)$ and 
$\bigl( \frac{ l- S_A^-(\eta)-(N-|A|) r \kappa a_N}
{\sqrt{(N-|A|) v_N^-}} \bigr)^2=o(1)$ 
uniformly in $L_N-t_N +1 \leq l \leq L_N$ 
and $\eta \in D_A(\lfloor Ra_N \rfloor)$. 
Then, by (\ref{b2}) and (\ref{b3}), we have
\begin{align*}
\lambda_{N, R} := 
\sup\limits_{ L_N-t_N +1 \leq l \leq L_N} 
\sup\limits_{\eta \in D_A(\lfloor Ra_N \rfloor)} 
\Bigm| \! 
\frac{
\overline{\nu}_{\phi^*_N, \psi^*_N}^{\Lambda_N \setminus A}
\bigl( S_{\Lambda_N \setminus A} = K_N-S_A(\eta), 
S_{\Lambda_N \setminus A}^- = l-S_A^-(\eta)\bigr)}
{\overline{\nu}_{\phi^*_N, \psi^*_N}^{N}
(\widetilde{\Omega}_N^{K_N, l} \bigr)}
-1 \! \Bigm| 
\to 0, 
\end{align*}
as $N\to \infty$ for every $R>0$. 
To summarize, 
\begin{align*}
I_5 
& \leq 
\lambda_{N, R} + 
C_0 \sum\limits_{\eta \in \mathbb{Z}^A \setminus 
D_A(\lfloor Ra_N \rfloor)} 
\overline{\nu}_{\phi^*_N, \psi^*_N}^{A}(\eta) 
\leq 
\lambda_{N, R} + 
C_0 |A| {\nu}_{\phi^*_N, \psi^*_N}( |\eta| \geq Ra_N). 
\end{align*}
The second term goes to $0$ 
as $N\to\infty$ and $R\to \infty$ by the 
estimate of $I_8$ above. 
Hence, we obtain $I_5 \to 0$ and we can complete the proof. 

\end{proof}

\subsection{Proof of Theorem \ref{thm1}; the case $g^+ > g^-$}
Next, we consider the case $g^+>g^-$. 
For $\phi \in (g^-, g^+)$, we have 
$Q_n^+(\phi) = \sum\limits_{k\geq 0}k^n 
\phi^k \frac{1}{g(k)!}<\infty$ and 
$\widetilde{Q}_n^-(\phi) := 
\sum\limits_{k\leq -1}|k|^n \phi^k {g(k+1)!}
=\sum\limits_{k\geq 1 }k^n 
\frac{1}{\phi^k} {g(-k+1)!}<\infty$ for every $n\in \mathbb{Z}_+$. 
Therefore, 
we can define a probability distribution 
${\nu}_\phi$ on $\mathbb{Z}$ by the following. 
\begin{equation}\label{nu0}
{\nu}_\phi(k) = \frac{1}{{Z}(\phi)}
\phi^k G(k),\ k\in \mathbb{Z},
\end{equation} 
where $G(k)$ is given by (\ref{margi}) and 
${Z}(\phi) =Q_0^+(\phi)+ \widetilde{Q}_0^-(\phi)$. 
For $A \subset \mathbb{N}$, 
$\overline{\nu}_{\phi}^A$ 
denote the product measure on 
$\mathbb{Z}^A$ whose one site marginal distribution equals to 
${\nu}_{\phi}$. 
We denote the corresponding measure as $\overline{\nu}_{\phi}^N$ 
when $A=\Lambda_N$. 
Then, we have the identity 
${\mu}_{N}^{K, L}(\, \cdot\,) = 
\overline{\nu}_{\phi}^{N}
\bigl(\, \cdot  \bigm| {\Omega}_N^{K, L} \bigr)$. 
The right-hand side does not depend on the 
choice of the parameter $\phi$. 

The asymptotics of $Q_n^+(\phi)$ as $\phi \uparrow g^+$ 
is given by (\ref{q+}) and 
$\lim\limits_{\phi \uparrow g^+}
\widetilde{Q}_n^-(\phi) 
=\sum\limits_{k\geq 1 }k^n 
\frac{1}{(g^+)^k} {g(-k+1)!} =: d_n <\infty$. 
It is easy to see that 
$E^{{\nu}_\phi}[\xi]$ is continuous, 
increasing in $\phi$ 
and diverges to infinity as $\phi \uparrow g^+$. 
Similarly to Lemma \ref{lem1}, we can prove the following. 
\begin{lem}\label{lem3}
Let $\kappa>0$ and 
$\{a_N\}_{N \geq 1}$ be an increasing 
sequence of positive numbers 
which satisfies $\lim\limits_{N\to \infty}a_N=\infty$. 
There exists a unique increasing sequence 
$\{\phi^*_N\}_{N \geq 1}$ which satisfies 
$E^{{\nu}_{\phi^*_N}}[\xi] = 
\kappa a_N$ for every $N\geq 1$. 
Moreover, 
$1-\frac{\phi^*_N}{g^+} \sim \frac{1}{\theta a_N}$ 
as $N\to \infty$ where $\theta= \frac{\kappa}{\rho+ 1}$. 
\end{lem}
\noindent
For this choice of $\phi$, we have the following. 
\begin{lem}\label{lem5}
Let $\{\phi^*_N\}_{N\geq 1}$ be a sequence obtained 
in Lemma \ref{lem3} and set $L_N= r\kappa a_N N$ for 
$r, \kappa>0$. 
There exists $\gamma_0>0$ such that 
$\overline{\nu}_{\phi^*_N}^N 
\bigl( S_N^- \geq L_N \bigr) \leq e^{-\gamma_0 L_N}$ 
for every $N$ large enough. 
\end{lem}
\begin{proof}
By Markov's inequality, we have 
\begin{align*}
\overline{\nu}_{\phi^*_N}^N 
\bigl( S_N^- \geq L_N \bigr) \leq e^{-\gamma L_N}
E^{\overline{\nu}_{\phi^*_N}^N }[e^{\gamma S_N^-}] 
= e^{-\gamma L_N}\bigl(E^{\nu_{\phi^*_N}}[e^{\gamma \xi^-}]\bigr)^N, 
\end{align*}
for every $\gamma>0$. We compute that 
\begin{align*}
E^{\nu_{\phi^*_N}}[e^{\gamma \xi^-}] 
& = 
\frac{1}{{Z}(\phi^*_N)} 
\Bigl\{ \sum\limits_{k\geq 0} (\phi^*_N)^k \frac{1}{g(k)!} 
+ \sum\limits_{k\leq -1}e^{\gamma(-k)} (\phi^*_N)^k {g(k+1)!}\Bigr\}\\
& = \frac{1}{Q_0^+(\phi^*_N) + \widetilde{Q}_0^-(\phi^*_N)}
\Bigl\{Q_0^+(\phi^*_N) + \sum\limits_{k\geq 1}
\bigl(\frac{e^\gamma}{\phi^*_N}\bigr)^k g(-k+1)!\Bigr\} \\
& \leq 
1+ \frac{1}{Q_0^+(\phi^*_N)}\sum\limits_{k\geq 1}
\bigl(\frac{e^\gamma}{\phi^*_N}\bigr)^k g(-k+1)!. 
\end{align*}
By the asymptotics of $\phi^*_N$ and the condition $(g)$, 
$\sum\limits_{k\geq 1}
\bigl(\frac{e^\gamma}{\phi^*_N}\bigr)^k g(-k+1)! 
\leq C\sum\limits_{k\geq 1}\bigl(\frac{e^{\gamma}g^-}{g^+}
\bigr)^k k^{\rho}$ for every $N$ large enough 
and this converges if $\gamma>0$ is small enough. 
Combining this with the asymptotics $Q_0^+(\phi^*_N) \sim B^+_0(\theta) 
(a_N)^{\rho +1}$, we obtain 
$E^{\nu_{\phi^*_N}}[e^{\gamma \xi^-}] \leq 1+ \frac{C}{(a_N)^{\rho +1}} 
\leq e^{\frac{C}{(a_N)^{\rho +1}}}$. 
Hence, we obtain 
$
\overline{\nu}_{\phi^*_N}^N 
\bigl( S_N^- \geq L_N \bigr) 
\leq e^{-\gamma L_N} 
e^{\frac{CN}{(a_N)^{\rho +1}}}$. 
\end{proof}

Next, similarly to the case $g^+ \leq g^-$, we formulate the 
local limit theorem. 
Let $\{\phi_N\}_{N \geq 1}$ be an 
increasing sequences of positive numbers 
which satisfies $\phi_N \in (g^-, g^+)$ 
and 
$1-\frac{\phi_N}{g^+} \sim \frac{1}{\lambda a_N}$
as $N\to \infty$ 
where $\lambda>0$ and 
$\{a_N\}_{N\geq 1}$ be an increasing 
sequence of positive numbers which satisfies 
$\lim\limits_{N\to \infty}a_N = \infty$. 
In this subsection, 
$\{{X}_j\}_{j\in \Lambda_N}$ denotes a family of 
independent and identically distributed 
$\mathbb{Z}$-valued random variables 
with common distribution ${\nu}_{\phi_N}$, 
$N\in \mathbb{N}$. 
By (\ref{q+}), we have 
$Q_n^+ (\phi_N) \sim B_n^+(\lambda) (a_N)^{\rho + n+1}$ and 
$\widetilde{Q}_n^- (\phi_N) \to d_n$ 
as $N\to \infty$ for every $n\in \mathbb{Z}_+$. 
This yields that 
\begin{align*}
{Z}(\phi_N)& 
= 
{{Q}_0^+(\phi_N)}+ {\widetilde{Q}_0^-(\phi_N)}
= 
{B_0^+(\lambda)}(a_N)^{\rho+1}(1+o(1)), \\
m_N & 
= E[{X}_1] 
= \frac{Q_1^+(\phi_N)-{\widetilde{Q}_1^-(\phi_N)}}{{Z}(\phi_N)} 
= \frac{B_1^+(\lambda)}{B_0^+(\lambda)}a_N(1+o(1)), \\
v_N & 
= \mathrm{Var}({X}_1) 
= \frac{Q_2^+(\phi_N) +{\widetilde{Q}_2^-(\phi_N)}}{{Z}(\phi_N)}
-m_N^2 
= \frac{B_2^+(\lambda)B_0^+(\lambda) -(B_1^+(\lambda))^2}
{(B_0^+(\lambda))^2}(a_N)^2(1+o(1)). 
\end{align*}
In this setting, 
note that all moments of the negative part of $X_1$ 
converge to $0$ as $N\to\infty$.
Then, we have the following local limit theorem 
for the distribution of 
$\sum\limits_{j \in \Lambda_N} X_j$. 
\begin{thm}\label{lclt3}
Assume that 
$a_N=o(N^\beta)$ as $N\to \infty$ for some 
$\beta >0$. 
Then, 
for every finite set $B\subset \mathbb{N}$, it holds that 
\begin{align*}
\lim\limits_{N\to \infty} 
\sup\limits_{K \in \mathbb{Z}} 
\Bigm|\! 
\sqrt{ (N-|B|) v_N} 
P \Bigl(
\sum\limits_{j \in \Lambda_N \setminus B} X_j =K 
\Bigr) -
q_0(\overline{K}_N)
\!\Bigm| =0, 
\end{align*}
where
$ 
q_0(x)= 
\frac{1}{\sqrt{2 \pi}}
\exp\Bigl\{-\frac{x^2}{2}\Bigr\} 
$ and 
$\overline{K}_N= \frac{K-(N-|B|)m_N}{\sqrt{(N-|B|)v_N}}$. 
\end{thm}

\smallskip

\begin{proof}[Proof of Theorem \ref{thm1}; the case $g^+ > g^-$]
Let $f:\mathbb{R}^A \to \mathbb{R}$ be a 
bounded continuous local function 
where $A$ is a finite subset of $\mathbb{N}$. 
For every $N\in \mathbb{N}$ so that $\Lambda_N \supset A$, 
we have 
\begin{align*}
& E^{\mu_{N}^{K_N, L_N}}\bigl[ f\bigl( 
\frac{\cdot}{a_N} \bigr)\bigr]  \\
& \quad = 
\sum\limits_{\eta\in \mathbb{Z}^A} 
f\bigl( \frac{\eta}{a_N} \bigr) 
{\mu_{N}^{K_N, L_N}}( \xi_A = \eta) \\
& \quad = 
\sum\limits_{\eta\in \mathbb{Z}^A} 
f\bigl( \frac{\eta}{a_N} \bigr) 
{\overline{\nu}^{N}_{\phi}}
\bigl(\xi_A = \eta \bigm| {\Omega}_N^{K_N, L_N} \bigr) \\
& \quad = 
\sum\limits_{\eta\in \mathbb{Z}^A} 
f\bigl( \frac{\eta}{a_N} \bigr) 
\overline{\nu}_{\phi}^{A}(\eta) 
+ 
\sum\limits_{\eta\in \mathbb{Z}^A} 
f\bigl( \frac{\eta}{a_N} \bigr) 
\overline{\nu}_{\phi}^{A}(\eta)
\Bigl\{
\frac{
\overline{\nu}_{\phi}^{\Lambda_N \setminus A}
\bigl( S_{\Lambda_N \setminus A} = K_N-S_A(\eta), 
S_{\Lambda_N \setminus A}^- \leq L_N-S_A^-(\eta)\bigr)}
{\overline{\nu}_{\phi}^{N}
({\Omega}_N^{K_N, L_N} \bigr)} -1 \Bigr\} \\
& \quad =: I_9 +I_{10}. 
\end{align*}
Since the right hand side does not depend on the 
choice of $\phi$, we choose $\phi^*_N$ so that 
$E^{\nu_{\phi^*_N}}[\xi]= \kappa a_N$. 

For the convergence of $I_9$, 
we again assume that 
$f$ is a function of one variable for notational simplicity. 
We have that 
\begin{align*}
\sum\limits_{\eta\in \mathbb{Z}} 
f\bigl( \frac{\eta}{a_N} \bigr) 
{{\nu}}_{\phi^*_N}(\eta) 
& = 
\sum\limits_{k \geq 0} 
f\bigl( \frac{k}{a_N} \bigr) \frac{1}{{Z}(\phi^*_N)}
(\phi^*_N)^k 
\frac{1}{g(k)!}
+ 
\sum\limits_{k\leq -1} 
f\bigl( \frac{k}{a_N} \bigr) \frac{1}{{Z}(\phi^*_N)}(\phi^*_N)^{k} 
{g(k+1)!}.
\end{align*}
By the similar computation as before, the first term converges to 
$\frac{1}{\Gamma(\rho +1)\theta^{\rho +1}}\int_0^\infty f(t) 
e^{-\frac{1}{\theta}t}t^{\rho}dt$ and 
the second term converges to $0$ as $N\to \infty$. 

To prove that $I_{10} \to 0$ as $N\to \infty$, 
it suffices to establish the following two estimates, 
as in the case $g^+ \leq g^-$. 
\begin{align}\label{C1}
\sup_{N\geq 1}
\sup\limits_{\eta \in \mathbb{Z}^A} 
\Bigm| \!
\frac{
\overline{\nu}_{\phi^*_N}^{\Lambda_N \setminus A}
\bigl( S_{\Lambda_N \setminus A} = K_N-S_A(\eta), 
S_{\Lambda_N \setminus A}^- \leq L_N-S_A^-(\eta)\bigr)}
{\overline{\nu}_{\phi^*_N}^{N}
({\Omega}_N^{K_N, L_N} \bigr)} -1 \!\Bigm| \leq C, 
\end{align}
for some $C<\infty$ and 
\begin{align}\label{C2}
\sup\limits_{\eta \in D_A(\lfloor Ra_N \rfloor)} 
\Bigm| \!
\frac{
\overline{\nu}_{\phi^*_N}^{\Lambda_N \setminus A}
\bigl( S_{\Lambda_N \setminus A} = K_N-S_A(\eta), 
S_{\Lambda_N \setminus A}^- \leq L_N-S_A^-(\eta)\bigr)}
{\overline{\nu}_{\phi^*_N}^{N}
({\Omega}_N^{K_N, L_N} \bigr)} -1 \!\Bigm| 
\to 0, 
\end{align}
as $N\to \infty$ for every $R>0$. 
We have that 
\begin{align*}
\overline{\nu}_{\phi^*_N}^{N}
({\Omega}_N^{K_N, L_N} \bigr) 
& = 
\overline{\nu}_{\phi^*_N}^{N}(S_N = K_N) - 
\overline{\nu}_{\phi^*_N}^{N}(S_N = K_N, S_N^- \geq L_N+1). 
\end{align*}
By Theorem \ref{lclt3} and the choice of $\phi_N^*$, 
\begin{align*}
\overline{\nu}_{\phi^*_N}^{N}(S_N = K_N) = 
\frac{1}{\sqrt{N v_N}}\bigl(\frac{1}{\sqrt{2\pi}} + o(1)\bigr).
\end{align*}
Also, by Lemma \ref{lem5}, 
\begin{align*}
\overline{\nu}_{\phi^*_N}^{N}(S_N = K_N, S_N^- \geq L_N+1) 
\leq 
\overline{\nu}_{\phi^*_N}^{N}(S_N^- \geq L_N) 
\leq e^{-\gamma_0 L_N}, 
\end{align*}
for every $N$ large enough and this is negligible 
compared to 
$\overline{\nu}_{\phi^*_N}^{N}(S_N = K_N)$. 
Therefore, we obtain 
\begin{align}\label{eq11}
\overline{\nu}_{\phi^*_N}^{N}
({\Omega}_N^{K_N, L_N} \bigr) 
& = 
\frac{1}{\sqrt{ 2\pi N v_N}}(1 + o(1)),
\end{align}
as $N\to \infty$. 
Next, applying Theorem \ref{lclt3} again, 
the following holds for every $\eta\in \mathbb{Z}^A$. 
\begin{align*}
& \overline{\nu}_{\phi^*_N}^{\Lambda_N \setminus A}
\bigl( S_{\Lambda_N \setminus A} = K_N-S_A(\eta), 
S_{\Lambda_N \setminus A}^- \leq L_N-S_A^-(\eta)\bigr) \\
& \quad \leq 
\overline{\nu}_{\phi^*_N}^{\Lambda_N \setminus A}
\bigl( S_{\Lambda_N \setminus A} = K_N-S_A(\eta)\bigr) \\
& \quad  = 
\frac{1}{\sqrt{(N-|A|)v_N}}\Bigl[
\frac{1}{\sqrt{2\pi}} 
\exp\bigl\{
-\frac{(K_N-S_A(\eta)-(N-|A|)m_N)^2}{2(N-|A|)v_N}
\bigr\} + o(1) 
\Bigr], 
\end{align*}
where the $o(1)$ term does not depend on $\eta$. 
Combining this with (\ref{eq11}), we obtain (\ref{C1}). 

Finally, for every $\eta \in D_A(\lfloor Ra_N \rfloor)$, 
we have 
$0 \leq S_A^-(\eta) \leq |A|R a_N = o(L_N)$. 
Therefore, by an argument similar to the above, 
\begin{align*}
& \overline{\nu}_{\phi^*_N}^{\Lambda_N \setminus A}
\bigl( S_{\Lambda_N \setminus A} = K_N-S_A(\eta), 
S_{\Lambda_N \setminus A}^- \leq L_N-S_A^-(\eta)\bigr) \\
& \ =  
\overline{\nu}_{\phi^*_N}^{\Lambda_N \setminus A}
\bigl( S_{\Lambda_N \setminus A} = K_N-S_A(\eta)\bigr) 
- \overline{\nu}_{\phi^*_N}^{\Lambda_N \setminus A}
\bigl( S_{\Lambda_N \setminus A} = K_N-S_A(\eta), 
S_{\Lambda_N \setminus A}^- \geq L_N-S_A^-(\eta)+1 
\bigr) \\
& \ = 
\frac{1}{\sqrt{(N-|A|)v_N}}\Bigl[
\frac{1}{\sqrt{2\pi}} 
\exp\bigl\{
-\frac{(|A|m_N-S_A(\eta))^2}{2(N-|A|)v_N}
\bigr\} + o(1) 
\Bigr] \\
& \ = 
\frac{1}{\sqrt{ 2\pi (N-|A|) v_N}}(1 + o(1)),
\end{align*}
where the $o(1)$ term does not depend on $\eta$. 
The last equality follows from 
$(|A|m_N-S_A(\eta))^2 \leq C (a_N)^2 = o(Nv_N)$. 
Combining this with (\ref{eq11}) again, 
we obtain (\ref{C2}) and we can complete the proof.  
\end{proof}

\subsection{Proof of Corollary \ref{cor1}}
\begin{proof}[Proof of Corollary \ref{cor1}]
By Proposition \ref{prop1} and Markov chain convergence theorem, 
we have 
\begin{align*}
\lim\limits_{n\to \infty} 
E_{\eta} \Bigl[ \frac{1}{N}
\bigm|\!\!\bigl\{x\in \Lambda_N; 
\frac{1}{a_N} {X_n (x)} \in (b, c) \bigr\}
\!\!\bigm|\Bigr] 
& = 
\lim\limits_{n\to \infty} 
\frac{1}{N} \sum\limits_{x\in \Lambda_N}
E_{\eta} \Bigl[ 
I \bigl( \frac{1}{a_N} X_n (x) 
\in (b, c)\bigr) \Bigr] \\
& = 
\frac{1}{N} 
\sum\limits_{x\in \Lambda_N}
\mu_{N}^{K_N, L_N}
\bigl( \frac{1}{a_N} \xi(x) \in (b, c)\bigr)\\
& = 
\mu_{N}^{K_N, L_N}
\bigl( \frac{1}{a_N} \xi(1) \in (b, c)\bigr), 
\end{align*}
for every $\eta \in \Omega_N^{K_N, L_N}$ and 
$N \in \mathbb{N}$. 
Therefore, it is sufficient to show that 
$\lim\limits_{N\to \infty} 
\mu_{N}^{K_N, L_N}
\bigl( \frac{1}{a_N} \xi(1) \in (b, c)\bigr) 
= 
\mu_{\kappa, \lambda} ((b, c))$. 
This follows from 
Theorem \ref{thm1} and 
the basic facts about the weak convergence of probability measures. 
\end{proof}

\section{Proof of the local limit theorem}
In this section we give the proof of 
Theorem \ref{lclt}. 
Let us recall the setting. 
$\{X_j\}_{j\in \Lambda_N}$ denotes a family of 
independent and identically distributed 
$\mathbb{Z}$-valued random variables 
with common distribution $\nu_{\phi_N, \psi_N}$ given by 
(\ref{nu}). 
$\{\phi_N\}_{N\geq 1}$ and 
$\{\psi_N\}_{N\geq 1}$ satisfy 
$1-\frac{\phi_N}{g^+} \sim \frac{1}{\lambda^+ a_N}$
and 
$1-g^-{\psi_N} \sim \frac{1}{\lambda^- a_N}$ 
as $N\to \infty$ where 
$\lambda^+, \lambda^- > 0$ and 
$\{a_N\}_{N\geq 1}$ be an increasing 
sequence of positive numbers which satisfies 
$\lim\limits_{N\to \infty}a_N = \infty$ and 
$a_N=o(N^\beta)$ as $N\to \infty$ for some 
$\beta >0$. 
We set 
$m_N^+ = E[X_1^{+}]$, 
$m_N^- = E[X_1^{-}]$, 
$v_N^+ = \mathrm{Var}(X_1^{+})$, 
$v_N^- = \mathrm{Var}(X_1^{-})$. 
These asymptotics are given by (\ref{mv}). 

We first note that Theorem \ref{lclt} is a 
local limit theorem for the joint distribution of
$\sum\limits_{j\in\Lambda_N}X_j^+$ 
and $\sum\limits_{j\in\Lambda_N}X_j^-$,
where the random variables form a triangular array, 
since their common distribution $\nu_{\phi_N,\psi_N}$ 
depends on $N$, the number of random variables. 
Although local limit theorems for triangular arrays of 
integer-valued random variables are classical, 
the known criteria for such local limit theorems 
(e.g. \cite{DM95}, \cite{M91}) do not apply in our setting, 
since both the mean and the variance under $\nu_{\phi_N,\psi_N}$ 
tend to infinity as $N\to\infty$. 
We prove the theorem following the classical argument 
in \cite[Chapter VII]{Pet}. 
The result is also a two-dimensional extension of 
\cite[Theorem 3.2]{S26}. 
\begin{proof}[Proof of Theorem \ref{lclt}]
For notational simplicity we only consider 
the case $B=\emptyset$. The modification for the general finite 
set $B\subset \mathbb{N}$ is straightforward. 

Set $\widetilde{X}_j^{+} := \frac{1}{\sqrt{N v_N^+}}
(X_j^{+} - m^+_N)$, 
$\widetilde{X}_j^{-} := \frac{1}{\sqrt{N v_N^-}}
(X_j^{-} - m^-_N)$, 
$j \in \Lambda_N$. 
$f(u_1, u_2) := 
E\bigl[\exp\{i (u_1 X_1^{+}+ u_2 X^-_1) \}\bigr]$, 
$(u_1, u_2) \in \mathbb{R}^2$ denotes 
the characteristic function of 
$(X_1^+, X_1^-)$ where $i= \sqrt{-1}$. 
We also define 
\begin{align*}
\widetilde{f}_N(u_1, u_2) 
& := 
E\bigl[\exp \bigl\{i 
(u_1 \sum\limits_{j \in \Lambda_N} \widetilde{X}_j^{+}
+ u_2 \sum\limits_{j \in \Lambda_N} \widetilde{X}_j^{-}
)\bigr\}\bigr] \\
& = 
\exp \Bigl\{ -i 
\bigl( u_1 \frac{\sqrt{N}{m_N^+}}{\sqrt{v_N^+}}
+u_2 \frac{\sqrt{N}{m_N^-}}{\sqrt{v_N^-}}
\bigr)\Bigr\}
\Bigl( f\bigl( 
\frac{u_1}{\sqrt{N v_N^+}}, 
\frac{u_2}{\sqrt{N v_N^-}} \bigr)\Bigr)^N. 
\end{align*} 
By the inversion formula, 
\begin{align*}
P \bigl(\sum\limits_{j\in \Lambda_N} X_j^{+} =K, 
\sum\limits_{j\in \Lambda_N} X_j^{-} =L
\bigr) 
& = 
\frac{1}{(2\pi)^2} 
\int_{-\pi}^{\pi} \int_{-\pi}^{\pi} 
e^{-i( K u_1  + L u_2 )} 
\bigl( f(u_1, u_2) \bigr)^N du_1 du_2, 
\end{align*}
for every $K, L\in \mathbb{Z}_+$. 
Therefore, we have 
\begin{align}\label{a1}
P \bigl(\sum\limits_{j\in \Lambda_N} X_j^{+} =K, 
\sum\limits_{j\in \Lambda_N} X_j^{-} =L
\bigr) 
& = 
\frac{1}{(2\pi)^2 N \sqrt{v_N^+ v_N ^-}} 
\int_{-\pi \sqrt{N v_N^+}}^{\pi\sqrt{N v_N^+}} 
\int_{-\pi \sqrt{N v_N^-}}^{\pi\sqrt{N v_N^-}} 
e^{-i(\overline{K}_N u_1 + \overline{L}_N u_2)} 
\widetilde{f}_N (u_1, u_2) du_1 du_2, 
\end{align}
where 
$\overline{K}_N = \frac{K-N m_N^+}{\sqrt{N v_N^+}} 
$ and 
$\overline{L}_N = \frac{L-N m_N^-}{\sqrt{N v_N^-}} 
$. 
Also, since 
$ q_0(x, y)= 
\frac{1}{2\pi \sqrt{1-\tau^2}}
\exp\Bigl\{-\frac{x^2 -2\tau xy+y^2}{2(1-\tau^2)}
\Bigr\}$ 
is the probability density function 
of the two-dimensional Gaussian vector with 
mean vector $\underline{0}=(0, 0)\in \mathbb{R}^2$ 
and covariance matrix 
$Q=\left(
\begin{array}{cc}
1 & \tau \\ 
\tau & 1\\ 
\end{array}
\right) 
$, we have 
\begin{align}\label{a2}
q_0(x, y) 
& = 
\frac{1}{(2\pi)^2} 
\int_{\mathbb{R}^2}
e^{-i(x u_1 + y u_2 )} 
\exp \bigl\{ 
-\frac{1}{2}(u_1^2 + 2\tau u_1 u_2 +u_2^2)
 \bigr\} du_1 du_2. 
\end{align}
Let $\gamma \in (0, \pi)$. 
By (\ref{a1}) and (\ref{a2}), 
\begin{align*}
& (2\pi)^2 \Bigm|\! N \sqrt{v_N^+ v_N ^-} 
P \bigl(\sum\limits_{j\in \Lambda_N} X_j^{+} =K, 
\sum\limits_{j\in \Lambda_N} X_j^{-} =L \bigr) 
- q_0(\overline{K}_N, \overline{L}_N) \!\Bigm| \\
& \qquad \leq 
\int_{D_1} 
|\widetilde{f}_N(u_1, u_2) 
-e^{-\frac{1}{2}(u_1^2+ 2\tau u_1 u_2+u_2^2)} | du_1 du_2 
+ 
\Bigl(\int_{D_2}+ \int_{D_3^+}+ \int_{D_3^-} \Bigr)
|\widetilde{f}_N(u_1, u_2) | du_1 du_2 \\
& \qquad \qquad 
+ 
\int_{D_4} e^{-\frac{1}{2}(u_1^2+ 2\tau u_1 u_2+u_2^2)} 
du_1 du_2 \\
& \qquad =: J_1+ J_2+J_3^+ + J_3^- +J_4, 
\end{align*}
where 
\begin{align*}
D_1 & = 
\{(u_1, u_2); \| \underline{u} \| \leq R\}, \\
D_2 & = 
\{(u_1, u_2); \| \underline{u} \| \geq R, 
|u_1| \leq \gamma \sqrt{N v_N^+}, 
|u_2| \leq \gamma \sqrt{N v_N^-}\}, \\
D_3^+ & = 
\{(u_1, u_2); 
\gamma \sqrt{N v_N^+} \leq |u_1| \leq \pi \sqrt{N v_N^+}, 
|u_2| \leq \gamma \sqrt{N v_N^-}\}, \\
D_3^- & = 
\{(u_1, u_2); 
|u_1| \leq \gamma \sqrt{N v_N^+}, 
\gamma \sqrt{N v_N^-} \leq |u_2| \leq \pi \sqrt{N v_N^-} 
\}, \\
D_4 & = \{(u_1, u_2); 
\| \underline{u} \| \geq R\}. 
\end{align*}
$\|\underline{u} \|$ denotes the Euclid norm of 
$\underline{u}= (u_1, u_2)\in \mathbb{R}^2$. 
Note that the right hand side does not depend on $K$ and $L$.
We show that the right-hand side converges to $0$ as 
$N\to \infty$ and $R\to \infty$ for $\gamma>0$ small enough. 

\medskip

\noindent
{\underline{{\em Estimate of $J_1$}}:} \ 
Assume that the law of 
$(\sum\limits_{j \in \Lambda_N} \widetilde{X}_j^{+}, 
\sum\limits_{j \in \Lambda_N} \widetilde{X}_j^{-})$ 
converges to 
the two-dimensional Gaussian vector 
$\mathcal{N}( \underline{0}, Q)$. 
Then, we have 
$\lim\limits_{N\to \infty} \widetilde{f}_N(u_1, u_2) 
= e^{-\frac{1}{2} (u_1^2+2\tau u_1 u_2+ u_2^2)}$ for every 
$(u_1, u_2) \in \mathbb{R}^2$ and we obtain 
$J_1 \to 0$ as $N\to \infty$ by the bounded convergence theorem 
for every $R>0$. 
For the convergence of the law of 
$(\sum\limits_{j \in \Lambda_N} \widetilde{X}_j^{+}, 
\sum\limits_{j \in \Lambda_N} \widetilde{X}_j^{-})$, 
we have only to show that 
$\sum\limits_{j\in \Lambda_N} 
E\bigl[
\| \underline{\widetilde{X}_j}\|^2 
I\bigl( \| \underline{\widetilde{X}_j}\| \geq \varepsilon 
\bigr)\bigr] 
= N
E\bigl[
\| \underline{\widetilde{X}_1}\|^2 
I\big( \| \underline{\widetilde{X}_1}\| \geq \varepsilon 
\bigr)\bigr] 
\to 0$ 
as $N\to \infty$ for every $\vare>0$ 
by Lindberg's central limit theorem (cf. \cite{Dur}) 
where 
$\underline{\widetilde{X}_j} = 
({\widetilde{X}}_j^+, {\widetilde{X}}_j^-)$, 
$j\in \Lambda_N$. 
We have that 
\begin{align*}
\| \underline{\widetilde{X}_1}\|^2
= | {\widetilde{X}_1^+}|^2 + | {\widetilde{X}_1^-}|^2 
\leq 
\frac{2}{N}\Bigl\{
\frac{(X_1^+)^2}{v_N^+} + \frac{(m_N^+)^2}{v_N^+} 
+ \frac{(X_1^-)^2}{v_N^-} + \frac{(m_N^-)^2}{v_N^-} 
\Bigr\}, 
\end{align*}
where we used the fact that 
$(a+b)^2 \leq 2a^2 +2b^2$ for every $a, b\in \mathbb{R}$. 
Also, if 
$\| \underline{\widetilde{X}_1}\| \geq \vare$ then 
$|\widetilde{X}_1^{+}| \geq \frac{1}{\sqrt{2}}\vare $ or 
$|\widetilde{X}_1^{-}| \geq \frac{1}{\sqrt{2}}\vare $ and taking 
$m_N^+ \asymp \sqrt{v_N^+}$, $m_N^- \asymp \sqrt{v_N^-}$ 
into account, 
this means that 
${X}_1^{+} \geq \frac{1}{\sqrt{2}}\vare \sqrt{Nv_N^+}$ or 
${X}_1^{-} \geq \frac{1}{\sqrt{2}}\vare \sqrt{N v_N^-}$ 
for every $N$ large enough. 
Therefore, 
\begin{align*}
N
E\bigl[
\| \underline{\widetilde{X}_1}\|^2 I\big(
 \| \underline{\widetilde{X}_1}\| \geq \varepsilon \bigr)\bigr] 
& \leq 
2E\Bigl[
\Bigl\{
\frac{(X_1^+)^2}{v_N^+} + \frac{(m_N^+)^2}{v_N^+} 
+ \frac{(X_1^-)^2}{v_N^-} + \frac{(m_N^-)^2}{v_N^-} 
\Bigr\}
I\bigl( {X}_1^{+} \geq \frac{1}{\sqrt{2}}\vare \sqrt{Nv_N^+} 
\bigr)\Bigr] \\
& \qquad + 
2E\Bigl[
\Bigl\{
\frac{(X_1^+)^2}{v_N^+} + \frac{(m_N^+)^2}{v_N^+} 
+ \frac{(X_1^-)^2}{v_N^-} + \frac{(m_N^-)^2}{v_N^-} 
\Bigr\}
I\bigl( {X}_1^{-} \geq \frac{1}{\sqrt{2}}\vare \sqrt{Nv_N^-} 
\bigr)\Bigr] \\
& \leq 
\frac{2}{v_N^+} E\bigl[(X_1^+)^2
I\bigl( {X}_1^{+} \geq \frac{1}{\sqrt{2}}\vare \sqrt{Nv_N^+} 
\bigr)\bigr] 
+ \frac{2}{v_N^-} E\bigl[(X_1^-)^2
I\bigl( {X}_1^{-} \geq \frac{1}{\sqrt{2}}\vare \sqrt{Nv_N^-} 
\bigr)\bigr] \\
& \qquad 
+ C P\bigl({X}_1^{+} \geq \frac{1}{\sqrt{2}}\vare \sqrt{Nv_N^+}
\bigr) 
+ C P\bigl({X}_1^{-} \geq \frac{1}{\sqrt{2}}\vare \sqrt{Nv_N^-}
\bigr) \\
& =: J_5^+ + J_5^- + J_6^+ + J_6^-, 
\end{align*}
where we used $X_1^+X_1^-=0$ for the last inequality. 
For the estimate of $J_5^+$, we have 
\begin{align*}
J_5^+ & 
= 
\frac{2}{v_N^+} 
\sum\limits_{k \geq \frac{1}{\sqrt{2}}\vare \sqrt{Nv_N^+}}
k^2 \frac{1}{Z(\phi_N, \psi_N)}(\phi_N)^k \frac{1}{g(k)!}\\
& \leq 
\frac{C}{(a_N)^{\rho+3}}  
\sum\limits_{k \geq C'\vare \sqrt{N}a_N} k^2 
\Bigl\{g^+ \bigl(1-\frac{1}{\lambda^+ a_N} \bigr)^k\Bigr\}
\frac{C^+ k^\rho}{(g^+)^k} + o(1) \\
& \leq 
\frac{C}{a_N}  
\sum\limits_{k \geq C'\vare \sqrt{N}a_N} 
(\frac{k}{a_N})^{\rho+2} 
e^{-\frac{1}{\lambda^+} \frac{k}{a_N}}
+ o(1), 
\end{align*}
for some constants $C, C'>0$ 
and every $N$ large enough 
where we used (\ref{mv}) and the condition  $(g)$. 
The right-hand side goes to $0$ as $N\to\infty$ 
because 
\begin{align*}
\frac{1}{a_N} 
\sum\limits_{ \frac{k}{a_N} \geq b}
\Bigl(\frac{k}{a_N}\Bigr)^{\rho+2} 
e^{-\frac{1}{\lambda^+} \frac{k}{a_N}}
\to 
\int_b^\infty t^{\rho +2}e^{-\frac{1}{\lambda^+}t}dt <\infty, 
\end{align*}
as $N\to \infty$ for every $b \geq 0$. 
By the similar computation we obtain 
$J_5^-\to 0$. 
For the estimate of $J_6^+$, we have 
\begin{align*}
J_6^+ \leq 
\frac{C}{\vare \sqrt{N v_N^+}} 
E[X_1^{+}] 
\leq 
\frac{C'}{\vare \sqrt{N}}
\to 0, 
\end{align*}
as $N\to \infty$ where we used 
Markov's inequality and (\ref{mv}). 
Similarly, $J_6^- \to 0$ as $N\to \infty$. 
Thus, we obtain 
$\lim\limits_{N\to \infty}
\sum\limits_{j\in \Lambda_N} 
E\bigl[
\| \underline{\widetilde{X}_j}\|^2 
I\bigl( \| \underline{\widetilde{X}_j}\| \geq \varepsilon 
\bigr)\bigr] =0$ for every $\varepsilon >0$. 

\medskip

\noindent
{\underline{{\em Estimates of $J_2$ and $J_4$}}:} \ 
By Taylor's theorem, there exists $\gamma_0>0$ such that 
for every $(u_1, u_2) \in \mathbb{R}^2$ which satisfies 
$\| (\frac{u_1}{\sqrt{N v_N^+}}, \frac{u_2}{\sqrt{N v_N^-}}) \| 
\leq \gamma_0$, we have 
\begin{align*}
\bigm|\! E
\bigl[e^{i (u_1 \widetilde{X}_1^{+}+ 
u_2 \widetilde{X}_1^{-})}\bigr]\! \bigm| 
& \leq 
1-\frac{1}{4}
\bigl(\frac{u_1}{\sqrt{N v_N^+}}, \frac{u_2}{\sqrt{N v_N^-}}
\bigr)
\left(
\begin{array}{cc}
\mathrm{Var}(X_1^+) & \mathrm{Cov}(X_1^+, X_1^-) \\ 
\mathrm{Cov}(X_1^+, X_1^-) & 
\mathrm{Var}(X_1^-) \\ 
\end{array}
\right) 
\left(
\begin{array}{c}
\frac{u_1}{\sqrt{N v_N^+}} \\
\frac{u_2}{\sqrt{N v_N^-}} 
\end{array}
\right)
\\
& = 
1-\frac{1}{4N}
\{ u_1^2 + 2\tau(1+o(1))u_1 u_2+u_2^2\}
\leq 
e^{-\frac{C}{N}(u_1^2+ u_2^2)}, 
\end{align*}
for some $C>0$ and every $N$ large enough. 
For the last inequality,  we used the fact that 
$u_1^2 +2\tau u_1u_2+u_2^2 \geq 
(1+\tau)(u_1^2+u_2^2)$ for 
$\tau \in (-1, 0)$ and 
every $u_1, u_2\in \mathbb{R}$. 
Therefore, by taking $\gamma$ in the definition of 
$D_2$ as $\frac{1}{\sqrt{2}}\gamma_0$, we have 
$|\widetilde{f}_N(u_1, u_2)| \leq e^{-C(u_1^2+u_2^2)}$ 
on $D_2$ and this yields that 
\begin{align*}
J_2 \leq 
\int_{\{\| \underline{u} \| \geq R\}} 
e^{-C(u_1^2+u_2^2)} du_1 du_2 
\to 0, 
\end{align*}
as $R\to \infty$. Similarly, we have 
$J_4\to 0$ as $R\to \infty$. 

\medskip

\noindent
{\underline{{\em Estimates of $J_3^+$ and $J_3^-$}}:} \ 
We take $\gamma$ in the definition of $J_3^+$ and $J_3^-$ as 
$\frac{1}{\sqrt{2}}\gamma_0$ above. 
Let $Y^+$ be a $\mathbb{Z}_+$-valued random variable whose distribution 
is given by 
\begin{equation}\label{nu1}
\nu^+(k):= P(Y^+=k) =\frac{1}{Q_0^+(\phi_N)}
(\phi_N)^k\frac{1}{g(k)!},\ k\geq 0, 
\end{equation}
and $Y^-$ be a $\mathbb{N}$-valued random variable whose distribution 
is given by 
\begin{equation}\label{nu2}
\nu^-(k):= P(Y^-=k) =\frac{1}{Q_0^-(\psi_N)}
(\psi_N)^k{g(-k+1)!},\ k\geq 1. 
\end{equation}
Then, we have 
\begin{align*}
f(u_1, u_2) 
= 
E\bigl[\exp\bigl\{i(u_1 X_1^+ +u_2X_1^-)\bigr\}\bigr] 
=
\frac{Q_0^+(\phi_N)}{Z_N(\phi_N, \psi_N)}E[e^{iu_1 Y^+}] 
+ \frac{Q_0^-(\psi_N)}{Z_N(\phi_N, \psi_N)}E[e^{iu_2 Y^-}], 
\end{align*}
and 
\begin{align*}
\frac{Q_0^+(\phi_N)}{Z_N(\phi_N, \psi_N)} \sim 
\frac{B_0^+(\lambda^+)}{B_0^+(\lambda^+)+B_0^-(\lambda^-)}, \ 
\frac{Q_0^-(\psi_N)}{Z_N(\phi_N, \psi_N)} \sim 
\frac{B_0^-(\lambda^-)}{B_0^+(\lambda^+)+B_0^-(\lambda^-)}, 
\end{align*}
as $N\to \infty$. 
If we can show that 
there exists $s \in (0, 1)$ such that 
$|E\bigl[\exp \bigl\{ i \frac{u_1}{\sqrt{N v_N^+}} Y^+\bigr\}\bigr]| 
\leq 1-s$ 
on $D_3^+$ for every $N$ large enough, 
then we have 
\begin{align*}
\Bigm| 
f\bigl( \frac{u_1}{\sqrt{N v_N^+}}, \frac{u_2}{\sqrt{N v_N^-}}
\bigr) \Bigm| 
\leq 
\frac{Q_0^+(\phi_N)}{Z_N(\phi_N, \psi_N)}(1-s)
+ \frac{Q_0^-(\psi_N)}{Z_N(\phi_N, \psi_N)}
\leq 1-s', 
\end{align*}
for some $s'=s'(\lambda^+, \lambda^-)>0$ on $D_3^+$ 
and this yields that 
\begin{align*}
J_3^+ \leq \int_{D_3^+}(1-s')^N du_1 du_2 
\leq C\sqrt{Nv_N^+}\sqrt{Nv_N^-}(1-s')^N 
\leq C'Na_N^2 (1-s')^N \to 0, 
\end{align*}
as $N\to \infty$. 
Similarly, we obtain $I_3^- \to 0$ once we can show that 
$|E\bigl[\exp \bigl\{ i \frac{u_2}{\sqrt{N v_N^-}} Y^-\bigr\}\bigr]| 
\leq 1-s$ for some $s \in (0, 1) $
on $D_3^-$. 

For the estimate of 
$|E\bigl[\exp \bigl\{ i \frac{u_1}{\sqrt{N v_N^+}} Y^+\bigr\}\bigr]|$, 
let $\delta \in (0, 1)$ be fixed and define 
$h^\delta(u_1) := 
E\bigl[ (\delta e^{iu_1})^{Y^+}\bigr] = 
\sum\limits_{k\geq 0} \delta^k e^{iu_1 k} 
{{\nu}^+(k)}$. By the proof of 
\cite[Lemma 5.4]{LSV96}, we know that 
\begin{align*}
|h^\delta(u_1)| \leq 
\frac{1}{|1-\delta e^{iu_1}|}\Bigl\{
{\nu}^+ (0)+ 
\sum\limits_{k\geq 0} 
|{{\nu}^+(k+1)}-{{\nu}^+(k)}|\Bigr\}.  
\end{align*}
We have 
${\nu}^+(0) = \frac{1}{Q_0^+(\phi_N)} \to 0$ 
as $N\to \infty$. 
Also, 
\begin{align*}
|{{\nu}^+(k+1)}-{{\nu}^+(k)}|
& = 
{\nu}^+(k) 
\bigm|\! \frac{\phi_N}{g(k+1)}-1\!\bigm|  \\ 
& = 
{\nu}^+(k) 
\Bigm|\! \frac{g^+\bigl(1-\frac{1+o(1)}{\lambda^+ a_N}\bigr)}
{g^+ \bigl(1-\frac{\rho}{k}+O(\frac{1}{k^2}) \bigr)}-1\!\Bigm| , 
\end{align*}
as $N\to \infty$ and $k\to \infty$ by the condition $(g)$. 
This yields that there exist $C_1>0$ such that 
$|{{\nu}^+(k+1)}-{{\nu}^+(k)}| \leq C_1 \nu^+(k)$ 
for every $N\geq 1$ and $k \geq 1$. 
Also, 
for every $\vare>0$ there exist $N_0\geq 1$ and $C_2>0$ 
such that 
$|{{\nu}^+(k+1)}-{{\nu}^+(k)}| \leq C \nu^+(k)\bigl(\frac{1}{a_N} 
+ \frac{1}{k}) \leq \frac{C_2}{a_N} \nu^+(k)
$ 
for every $N\geq N_0$ and $k\geq \vare a_N$. 
Therefore, 
\begin{align*}
\sum\limits_{k\geq 0} 
|{{\nu}^+(k+1)}-{{\nu}^{+}(k)}| 
\leq 
C_1 \sum\limits_{k=0}^{\vare a_N} {\nu}^+ (k) 
+ 
\frac{C_2}{a_N} 
\sum\limits_{k\geq \vare a_N} {\nu}^+(k), 
\end{align*}  
for every $N$ large enough. 
The first term 
of the right-hand side converges to 
$C_3\int_0^{\vare} t^{\rho}e^{-\frac{1}{\lambda^+}t}dt$ 
for some $C_3>0$ by the similar computation as before. 
The second term is less than $\frac{C_2}{a_N}$ 
and this goes to $0$ as $N\to \infty$. 
As a result, for every $\vare> 0$ there exists $C_4(\vare)>0$ 
which satisfies $C_4(\vare) \to 0$ as $\vare \to 0$ and 
it holds that 
\begin{align*}
|h^\delta(u_1)| \leq 
\frac{1}{|1-\delta e^{i u_1}|} \bigl\{ C_4(\vare) + o(1)
\bigr\}, 
\end{align*}
for every $u_1 \in \mathbb{R}$. 
By taking the limit $\delta \uparrow 1$ and using the estimate 
$|1-e^{i u}| \geq \frac{2}{\pi}|u|$ for every 
$|u|\leq \pi$, we obtain 
\begin{align*}
|E\bigl[\exp \bigl\{ i \frac{u_1}{\sqrt{N v_N^+}} Y^+\bigr\}\bigr]|
& \leq 
\frac{\sqrt{N v_N^+}}{|u_1|} \frac{\pi}{2} 
\bigl\{ C_4(\vare) + o(1)\bigr\} \\
& \leq 
\frac{\pi}{\sqrt{2}\gamma_0}\bigl\{C_4(\vare)+o(1)\bigr\}, 
\end{align*}
for every $u_1 \in \mathbb{R}$ so that 
$\frac{1}{\sqrt{2}}\gamma_0 \sqrt{N v_N^+} \leq |u_1| \leq 
\pi \sqrt{N v_N^+}$. 
Hence, by taking $\vare >0$ small enough, 
there exists $s \in (0, 1)$ such that 
$|E\bigl[\exp \bigl\{ i \frac{u_1}{\sqrt{N v_N^+}} Y^+\bigr\}\bigr]|
\leq 1-s$ on $D_3^+$ for every $N$ large enough. 

The estimate of 
$|E\bigl[\exp \bigl\{ i \frac{u_1}{\sqrt{N v_N^-}} Y^-\bigr\}\bigr]|$ 
can be obtained by the similar manner. We only state the difference. 
$\nu^+$ in the above argument is replaced by $\nu^-$ and we have 
\begin{align*}
|{{\nu}^-(k+1)}-{{\nu}^-(k)}|
& = 
{\nu}^-(k) 
\bigm|\! {\psi_N}{g(-k)}-1\!\bigm|  \\ 
& = 
{\nu}^-(k) 
\Bigm|\! \frac{1}{g^-}\bigl(1-\frac{1+o(1)}{\lambda^- a_N}\bigr)
{g^- \bigl(1+\frac{\rho}{k}+O(\frac{1}{k^2}) \bigr)}-1\!\Bigm|, 
\end{align*}
as $N\to \infty$ and $k\to \infty$ by the condition $(g)$. 
This yields that there exist $C'_1>0$ such that 
$|{{\nu}^-(k+1)}-{{\nu}^-(k)}| \leq C'_1 \nu^-(k)$ 
for every $N\geq 1$ and $k \geq 1$. 
Also, 
for every $\vare>0$ there exist $N'_0\geq 1$ and $C'_2>0$ 
such that 
$|{{\nu}^-(k+1)}-{{\nu}^-(k)}| \leq C \nu^-(k)\bigl(\frac{1}{a_N} 
+ \frac{1}{k}) \leq \frac{C'_2}{a_N} \nu^-(k)
$ 
for every $N\geq N_0$ and $k\geq \vare a_N$. 
Otherwise, the details are the same as above.

\end{proof}

\medskip

Finally, Theorem \ref{lclt3} is the one-dimensional version of 
Theorem \ref{lclt} and has a statement similar to that of
\cite[Theorem 3.2]{S26}. 
Note that all moments of 
the negative part under $\nu_{\phi_N}$ 
converge to $0$ as $N\to\infty$. 
Therefore, 
Theorem \ref{lclt3} can be proved by exactly the same argument as 
that used in the proof of \cite[Theorem 3.2]{S26}, 
and we omit the details.


\section*{Acknowledgement 
}
This work was partially supported 
by JSPS KAKENHI Grant Number 26K06851. 
Generative AI tools were used for language editing 
as well as to debug and optimize the code used in the numerical simulations. 
After using these tools, the author reviewed and edited the content 
and source code as needed, 
and takes full responsibility for the content of the article.


\end{document}